\documentclass[a4paper,reqno,oneside]{amsart}

\usepackage{amsmath,amsfonts,amssymb,amsthm}
\usepackage{mathtools}
\usepackage{comment}
\usepackage[colorinlistoftodos, textwidth=3.1cm]{todonotes}
\usepackage{indentfirst}
\usepackage{geometry}

\usepackage[utf8]{inputenc} 
\usepackage[T1]{fontenc}    
\usepackage[hidelinks]{hyperref}       
\usepackage{url}            
\usepackage{booktabs}       
\usepackage{nicefrac}       
\usepackage{microtype}      
\usepackage{xcolor}         
\usepackage{cleveref}       
\usepackage{lipsum}         
\usepackage{graphicx}
\usepackage{doi}
\usepackage{booktabs}		

\usepackage{ifthen}

\usepackage{tikz}
\usepackage{tikzscale}
\usepackage{pgfplots}
\usepackage{float}
\usepackage{placeins} 

\theoremstyle{plain}
\newtheorem{teo}{Theorem}
\newtheorem{prop}[teo]{Proposition}
\newtheorem{corl}[teo]{Corollary}
\newtheorem{lema}[teo]{Lemma}
\newtheorem{remk}[teo]{Remark}
\theoremstyle{definition}

\newtheorem*{teo*}{Theorem}

\newcommand{\rquad}[1]{\foreach \n in {1,...,#1}{\quad}}

\title{Analysis of an Emden-Fowler-Hénon type equation}

\author[P.S. Corrêa Junior]{Pablo dos Santos Corrêa Junior}
\address[P.S. Corrêa Junior]{Instituto de Ciências Matemáticas e de Computação \\
		Universidade de São Paulo\\
		São Carlos, SP, Brazil.}
		\email{p.correa@usp.br}
		
\author[E. Moreira dos Santos]{Ederson Moreira dos Santos}
\address[E. Moreira dos Santos]{Instituto de Ciências Matemáticas e de Computação \\
	Universidade de São Paulo\\
	São Carlos, SP, Brazil.}
	\email{ederson@icmc.usp.br}
\author[D. Schiera]{Delia Schiera}
\address[D. Schiera]{CAMGSD, Departamento de Matemática,\\
		Instituto Superior Técnico\\
		Universidade de Lisboa \\
		Lisboa, Portugal.}
		\email{delia.schiera@tecnico.ulisboa.pt}
\hypersetup{
	pdftitle={Analysis of an Emden-Fowler-Hénon type equation},
	pdfsubject={MSC: 35B06, 35B40, 35J15, 35J61},
	pdfauthor={Pablo dos Santos Correa Junior, Ederson Moreira dos Santos, Delia Schiera},
	pdfkeywords={Elliptic equations, Radial solutions, Qualitative properties},
}

\begin{document}
	\maketitle

	\begin{abstract}
		This work investigates a Hénon-type equation with the cap-shaped weight $
		V_{\alpha}(|x|)=(4|x|(1-|x|))^\alpha$,
		which concentrates on the sphere of radius \(1/2\) as \(\alpha\to\infty\). Particular attention is devoted to ground state solutions and the phenomenon of symmetry breaking in the large-\(\alpha\) regime. A complete description of the asymptotic behavior of ground state radial solutions is obtained as \(\alpha\to\infty\). The analytical results are further supported by numerical simulations of the radial ground states.
	\end{abstract}
	
    
	\noindent {\footnotesize{\textit{\textbf{Keywords:}} Asymptotic behavior, Elliptic equations, Hénon type equations, Numerical approximation, Qualitative properties, Radial solutions, Symmetry breaking.		}}
	
	\section{Introduction}

	We are interested in \emph{positive} solutions to the problem
	\begin{equation}\label{problema-principal}
		\begin{cases}
			-\Delta u &= (4|x|(1-|x|))^\alpha |u|^{p-1}u \quad \text{in } B, \\
			\phantom{-\Delta }u&=0  \rquad{9} \;\, \text{on } \partial B,
		\end{cases}
	\end{equation}
	with $\alpha > 0$, $B= \{x\in \mathbb{R}^N; |x|< 1\}$ and $1<p$. Here we set
	\[
	\begin{cases}
		2^*=\frac{2N}{N-2} \quad \text{for} \quad N\geq 3 \quad \text{and}  \quad 2^* =+ \infty \quad \text{for} \quad N=1, 2,\vspace{5pt}\\
		2^*_{\alpha}=\frac{2(N+\alpha)}{N-2} \quad \text{for} \quad N\geq 3 \quad \text{and}  \quad 2^*_{\alpha} =+ \infty \quad \text{for} \quad N=1, 2.
	\end{cases}
	\]
	Problem \eqref{problema-principal} is a particular case of the class of problems given by $\Delta u + V(|x|)|u|^{p-1}u = 0$ with Dirichlet boundary conditions. A wide variety of weights $V(|x|)$ have been studied in the literature, especially in cases where $V\equiv 1$ and $V(|x|)=|x|^\alpha$. When considering radial solutions, equation \eqref{problema-principal} is also an Emden-Fowler equation, in the sense defined by \cite{Wong1975}.
	
	If $V \equiv 1$, the resulting equation is known as the \emph{Lane–Emden equation}, which has numerous applications. In particular, it is used in astronomy \cite[Sec.~7.2]{Hansen2004} to model stellar structures, as well as in various problems in analysis and geometry; see \cite{Talenti1976,Gidas1981,Dancer2012,Amadori2017,Ioku2018} and the references therein. 
	
	If $V(|x|) = |x|^\alpha$, the equation is referred to as the \emph{Hénon equation}, introduced in \cite{Henon1974} as a generalization of the Lane-Emden model. In that work, the author studied the behavior of spherical stellar systems, with particular emphasis on radial solutions. Within this framework, natural questions concerning the stability of the model arise: for which values of $\alpha$ and $p$ does a radial solution of $\Delta u + |x|^\alpha |u|^{p-1}u = 0$ correspond to a stable stellar structure?
	From both physical and variational viewpoints, one expects a natural candidate to be a radial solution that is also a \emph{ground state}, namely, a solution that minimizes the associated energy functional among all admissible nontrivial solutions. However, as shown in \cite{SMETS2002}, when $\alpha > 0$ is sufficiently large, ground state solutions of the Hénon equation fail to be radial; this phenomenon is known as \emph{symmetry breaking}.
	This loss of radial symmetry has motivated extensive research on the \emph{asymptotic behavior} of both ground state solutions and radial ground state solutions of the Hénon equation, particularly as $\alpha \to \infty$ or as $p$ approaches its limiting values; see \cite{SMETS2002,Cao2008,Byeon2006,LeitedaSilva2021} for further details.
	
	Building on the extensive literature on the Hénon equation, one may select the weight $V$ in $\Delta u + V(|x|)|u|^{p-1}u = 0$ so as to capture behaviors complementary to those induced by $|x|^\alpha$, while addressing the same fundamental questions: \emph{Are ground state solutions radial? What is the asymptotic behavior of ground state solutions?}
	It is worth noting that if $V(|x|)$ is radially decreasing, then the moving planes method \cite{Gidas1979} ensures that all positive solutions of $\Delta u + V(|x|)|u|^{p-1}u = 0$ are radially symmetric. Therefore, the most interesting cases arise when $V$ is not radially decreasing.
	In \cite{Mercuri2019}, the authors considered a ``cup-shaped'' weight $V(|x|)$ defined by $(1 - \frac{|x|}{R})^\alpha$ for $0 \leq |x| \leq R$ and $(1 - \frac{1 - |x|}{1 - R})^\alpha$ for $R \leq |x| \leq 1$, and investigated the occurrence of symmetry breaking. Other classes of weights have been studied in \cite{Kajikiya2012}, where sufficient conditions for symmetry breaking are established.
	Motivated by the goal of complementing the existing literature through the exploration of different qualitative behaviors, we propose problem \eqref{problema-principal} with the ``cap-shaped'' weight $V(|x|) = (4|x|(1 - |x|))^\alpha$. In particular, this weight does not satisfy the assumptions of \cite[Theorem 2.1]{Kajikiya2012}, namely those involving the $\mu$-function introduced therein.

	From now on, we set $V_\alpha(r) = (4r(1-r))^\alpha$ for $r \in [0,1]$. Then the equation in \eqref{problema-principal} can be rewritten as
	\[
	\frac{1}{V_\alpha(|x|)} \Delta u + |u|^{p-1}u = 0,
	\]
	where $\frac{1}{V_\alpha(|x|)}$ may be interpreted as a diffusion coefficient. In this interpretation, the medium $B$ favors diffusion near the origin and near the boundary $\partial B$, while diffusion is minimal on the radius sphere $1/2$. This suggests that concentration phenomenon is likely to occur around this sphere as $\alpha \to \infty$. This observation provides one of the main motivations for considering weights of the form $V_\alpha$.
	\subsection{Basic notation}

	To state the main results we fix the following notation. Associated with problem \eqref{problema-principal}, the energy functional is defined by
	\begin{equation}\label{funcional energia}
		I(u)= \frac{1}{2}\int_B |Du|^2  dx - \frac{1}{p+1}\int_B V_\alpha(|x|) |u|^{p+1}  dx.
	\end{equation}
	For $1 < p < 2^* - 1$, critical points of $I : H^1_0(B) \to \mathbb{R}$ are classical solutions of \eqref{problema-principal}. For $1 < p < 2^*_{\alpha} - 1$, critical points of $I : H^1_{0,\mathrm{rad}}(B) \to \mathbb{R}$ correspond to classical radial solutions of \eqref{problema-principal}; see \cite{Ni1982}.
	Define
	\begin{equation}\label{minimizacao no espaco todo}
		S_\alpha = \inf_{u \in H^1_0(B)\setminus\{0\}}
		\frac{\int_B |Du|^2 \, dx}
		{\left(\int_B V_\alpha(|x|)\, |u|^{p+1} \, dx\right)^{\frac{2}{p+1}}}
	\end{equation}
	and
	\begin{equation}\label{minimizacao nas radiais}
		S_{\alpha,\mathrm{rad}} = \inf_{u \in H^1_{0,\mathrm{rad}}(B)\setminus\{0\}}
		\frac{\int_B |Du|^2 \, dx}
		{\left(\int_B V_\alpha(|x|)\, |u|^{p+1} \, dx\right)^{\frac{2}{p+1}}}.
	\end{equation}
	Also set
	\begin{equation*}
		S = \inf_{u \in H^1_0(B)\setminus\{0\}}
		\frac{\int_B |Du|^2 \, dx}
		{\left(\int_B |u|^{p+1} \, dx\right)^{\frac{2}{p+1}}}.
	\end{equation*}
	Solutions to the minimization problems \eqref{minimizacao no espaco todo} and \eqref{minimizacao nas radiais}, for $1 < p < 2^* - 1$ and $1 < p < 2^*_{\alpha} - 1$, respectively, yield, up to rescaling, ground state solutions and radial ground state solutions of \eqref{problema-principal}. Indeed, denoting by $U_\alpha$ and $U_{\alpha, rad}$ positive minimizers to $S_\alpha$ and $S_{\alpha,rad}$, respectively, with $\int_B V_\alpha(|x|)|U_{\alpha}|^{p+1}dx = \int_B V_\alpha(|x|)|U_{\alpha,rad}|^{p+1}dx = 1$, define
	\[
	u_\alpha = S_\alpha^{\frac{1}{p-1}} U_\alpha,
	\qquad
	u_{\alpha,\mathrm{rad}}=S_{\alpha,\mathrm{rad}}^{\frac{1}{p-1}} U_{\alpha,\mathrm{rad}}.
	\]
	
	Then $u_\alpha$ is a positive ground state solution, and $u_{\alpha,\mathrm{rad}}$ is a positive ground state radial solution of \eqref{problema-principal} (solution with minimal energy among nontrivial radial solutions).
	
	Throughout this text the notation $u_{\alpha}$ is used to denote a ground state solution of \eqref{problema-principal}, while $u_{\alpha,\mathrm{rad}}$ denotes a ground state \emph{radial} solution (i.e., $u_{\alpha,\mathrm{rad}} \in H^1_{0,\mathrm{rad}}(B)$) of \eqref{problema-principal}.
	Define $C_\alpha = I(u_\alpha)$ and $C_{\alpha,\mathrm{rad}} = I(u_{\alpha,\mathrm{rad}})$. A direct computation yields the following useful relations:
	\begin{equation}\label{eq util relacao entre S _alpha,rad e C_alpha,rad}
		S_{\alpha} = \left(\frac{2(p+1)}{p-1}\right)^{\frac{p-1}{p+1}} (C_{\alpha})^{\frac{p-1}{p+1}}
		\quad \text{and} \quad
		S_{\alpha,\mathrm{rad}} = \left(\frac{2(p+1)}{p-1}\right)^{\frac{p-1}{p+1}} (C_{\alpha,\mathrm{rad}})^{\frac{p-1}{p+1}}.
	\end{equation}
	
	The $C^2(\overline{B})$ regularity of $u_{\alpha}$ and $u_{\alpha,\mathrm{rad}}$ follows from classical elliptic regularity theory; see \cite[Appendix~B]{Struwe2010}.
	Moreover, ground state solutions of \eqref{problema-principal} exhibit a well-known symmetry property, they are \emph{foliated Schwarz symmetric}; see, for instance, \cite[Theorem~6.2]{Bartsch2005}.
	
	Let $G_N(t,s)$ denote the Green’s function associated with the problem
	\begin{equation}\label{Greenint}
		-\frac{1}{r^{N-1}}\bigl(r^{N-1}u'(r)\bigr)' = F(r), \quad r \in (0,1),
		\qquad \text{with } u'(0)=0, \; u(1)=0,
	\end{equation}
	where $F:(0,1) \to \mathbb{R}$. Then, see for instance \cite[Appendix]{Singh2014}, $G_N(t,s)$ has the following expression: if $N =1$, 
	\begin{equation*}
		G_N(t,s) = \begin{cases*}
			(1-s), \quad 0<t\leq s \leq 1,\\
			(1-t), \quad 0<s\leq t \leq 1; 
		\end{cases*}
	\end{equation*}
	if $N=2$,
	\begin{equation*}
		G_N(t,s) = -\begin{cases*}
			\ln (s), \quad 0<t\leq s \leq 1,\\
			\ln (t), \quad 0<s\leq t \leq 1; 
		\end{cases*}
	\end{equation*}
	and, for $N>2$, 
	\begin{equation*}
		G_N(t,s) = \frac{1}{N-2}\begin{cases*}
			\frac{1}{s^{N-2}}-1, \quad 0<t\leq s \leq 1,\\
			\frac{1}{t^{N-2}}-1, \quad 0<s\leq t \leq 1 .
		\end{cases*}
	\end{equation*}
	Therefore, any radial solutions of problem \eqref{problema-principal} can be expressed as	
	\begin{equation}
		\label{eq: green function integral representation}
		u(t) = \int_{0}^{1}s^{N-1}G_N(t,s)V_\alpha(s)u^p(s)ds.
	\end{equation}

	\subsection{Main results}
	In the search for ground state solutions to the problem \eqref{problema-principal}, the usual subcritical growth condition $p < 2^* - 1$ has to be imposed. With regard to the existence of radial solutions, the presence of the term $|x|^\alpha$ in the weight $V_{\alpha}$ allows the existence of solutions to a wider range of $p$, namely, for $1<p<2^*_\alpha-1$. The same phenomenon is observed for the Hénon equation in  \cite{Ni1982}.
	
	\begin{teo}[Existence of Minimizers]
		\label{teo existencia de minimizantes}
		
		Let $\alpha>0$ and $N\geq 1$. 
		\begin{itemize}
			\item[(i)] For $1<p<2^*_{\alpha}-1$, the constant $S_{\alpha,rad}$ is attained. Any minimizer for $S_{\alpha, rad}$ has definite sign. In particular, problem \eqref{problema-principal} has a positive radial classical solution.
			\item[(ii)] For $1<p<2^*-1$, the constant $S_{\alpha}$ is attained. Any minimizer for $S_{\alpha}$ has definite sign. In particular, problem \eqref{problema-principal} has a ground state solution. Moreover, any ground state solution of problem \eqref{problema-principal} has definite sign.
		\end{itemize}
	\end{teo}
	The proof of item $(i)$ is parallel to the proof of \cite[Theorem~6]{Ni1982}. As for item $(ii)$, by Rellich-Kondrachov Theorem (see \cite[Theorem~9.16]{2011_Brezis_BOOK}) there exists $U_\alpha \in H_0^1(B)$ minimizer for $S_\alpha$. Replacing $U_\alpha$ by $|U_\alpha|$ we may assume that the minimizer is non-negative. Now, by Lagrange multiplier rule and regularity theory, one obtains that $U_\alpha$ is $C^2(\overline{B})$ and thus, by the strong maximum principle, $U_\alpha$ is positive.

	The following theorem establishes a symmetry-breaking result for ground state solutions in the regime of large values of $\alpha$.
	
	\begin{teo}[Symmetry breaking]\label{teorema quebra de simetria}
		Let $N\geq 3$ and $\frac{N}{N-2}<p<2^*-1$. Then, for any $\alpha$ sufficiently large, ground state solutions of \eqref{problema-principal} are not radially symmetric. In particular, problem \eqref{problema-principal} has two positive classical solutions.    
	\end{teo}

	Thus, ground state solutions of problem \eqref{problema-principal} are not, in general, radially symmetric. Nevertheless, as previously observed, they exhibit foliated Schwarz symmetry, which, together with the following additional information on their geometric structure, allows for a more precise description of their shape.

	\begin{teo}\label{teorema solucoes decrescem fora da bola de raio meio}
		For $\alpha>0$, $N\geq 1$ and $1<p$, let $u$ be a $C^2(\overline{B})$ positive solution of problem \eqref{problema-principal}. Then $u$ is strictly radially decreasing in $B\backslash B_{\frac{1}{2}}(0)$ and $u$ attains its maximum in $B_{\frac{1}{2}}(0)$.
	\end{teo}
	
	Concerning ground state radial solutions, we obtain a sharp characterization of their asymptotic behavior as $\alpha \to \infty$. Set
	\begin{equation*}
		\tilde{C} = \begin{cases}
			\left(\frac{2^N(N-2)}{\sqrt{\pi}(2^{N-2}-1)}\right)^{\frac{1}{p-1}} &\text{if } N> 2,\\
			\left(\frac{4}{\sqrt{\pi}\ln(2)}\right)^{\frac{1}{p-1}}&\text{if } N=2,\\
			\left(\frac{4}{\sqrt{\pi}}\right)^{\frac{1}{p-1}}&\text{if } N=1
		\end{cases}
		\quad \text{and} \quad
		\eta = \begin{cases}
			\left(\frac{2^N}{\sqrt{\pi}}\left(\frac{N-2}{2^{N-2}-1}\right)^p \right)^{\frac{1}{p-1}} &\text{if }  N>2,\\
			\left(\frac{4}{\sqrt{\pi}\ln(2)^p}\right)^{\frac{1}{p-1}}&\text{if }  N=2,\\
			\left(\frac{2^{p+1}}{\sqrt{\pi}}\right)^{\frac{1}{p-1}}&\text{if } N=1.
		\end{cases}
	\end{equation*}
	
	\begin{teo}(Asymptotic behavior of ground state radial solutions)\label{teorema comport assintotico sol radial} Let $N\geq 1$ and $1<p<2^*_{\alpha}-1$. Then \begin{equation}\label{teoderivada1}\lim_{\alpha\rightarrow\infty}-\frac{u'_{\alpha,rad}(1)}{\alpha^{\frac{1}{2(p-1)}}}= \eta .\end{equation}
		Moreover, as $\alpha\rightarrow\infty$, 
		\begin{itemize}
			\item[i)] For $t\in[0,\frac{1}{2}]$ \[ \frac{u_{\alpha,rad}(t)}{\alpha^{\frac{1}{2(p-1)}}} \rightarrow \tilde{C}, \]
			
			\item[ii)] For $t\in[\frac{1}{2},1]$
			\[ \frac{u_{\alpha,rad}(t)}{\alpha^{\frac{1}{2(p-1)}}} \rightarrow \begin{cases}
				\frac{\eta}{N-2}\left(\frac{1}{t^{N-2}}-1\right)\quad &\text{if }N>2,\\
				-\eta\ln(t)\quad &\text{if }  N=2,\\
				\eta(1-t)\quad &\text{if } N=1.
			\end{cases} \]	 	
		\end{itemize}
		In addition, the convergence of $\dfrac{u_{\alpha,rad}(t)}{\alpha^{\frac{1}{2(p-1)}}}$ is uniform on $[0,1]$.
		
	\end{teo}
	
	The convergence stated in Theorem \ref{teorema comport assintotico sol radial} can be reformulated in terms of the values of $G_N(t,1/2)$.

	\begin{corl}\label{corl: escrevendo a convg em termos da funcao de green}
		Let $N\geq 1$ and $1<p<2^*_{\alpha}-1$. Then, as $\alpha\rightarrow \infty$, it holds that
		
		\begin{equation}\label{limitGreen}
			\frac{u_{\alpha,rad}}{\alpha^{\frac{1}{2(p-1)}}}(t) \rightarrow \eta \,G_N\!\left(t,\frac{1}{2}\right)  \quad \text{uniformly on} \quad [0,1].
		\end{equation}
	\end{corl}
	
	\medbreak
	\begin{remk} 
		The convergence results in Theorem \ref{teorema comport assintotico sol radial} and Corollary \ref{corl: escrevendo a convg em termos da funcao de green} deserve some further discussion. It is worth emphasizing that the maximum of the weight appearing in \eqref{problema-principal} is attained at $|x|=1/2$, which plays a central role in the asymptotic behavior of ground state radial solutions. In particular, the normalized function $u_{\alpha,\mathrm{rad}}(t)\tilde{C}^{-1}\alpha^{-\frac{1}{2(p-1)}}$ becomes asymptotically flat on $[0,1/2]$ as $\alpha \to \infty$.
		A similar phenomenon has been observed in the context of the Hénon equation, where the weight attains its maximum at $|x|=1$, and the corresponding normalized ground state radial solutions become flat uniformly on any compact subset of $B$; see \cite[Theorem~2.6(ii)]{LeitedaSilva2021}. In the present setting, the main additional feature lies in the description of the behavior on the interval $[1/2,1]$, which is fully captured by formula \eqref{limitGreen} and illustrated by the  numerical results and Figure \ref{fig:solucao diferentes alphas} at Section \ref{sec: analise numerica}.
	\end{remk}
	
	As a consequence of Theorem \ref{teorema comport assintotico sol radial} we also characterize $S_{\alpha,rad}$.
	\begin{teo}\label{teo caracterizacao de S alpha rad}
		Let $N\geq 1$ and $1<p<2_{\alpha}^*-1$. Then
		\begin{equation*}
			\lim_{\alpha\rightarrow \infty}\frac{S_{\alpha,rad}}{\alpha^{\frac{1}{p+1}}} = C
		\end{equation*}
		where
		\begin{equation*}
			C = \begin{cases}
				\left(\frac{N-2}{2^{N-2}-1}\right)\left(\frac{4^N}{\pi}|\mathbb{S}^{N-1}|^{p-1}\right)^{\frac{1}{p+1}},\quad &\text{if } N>2\\
				\frac{1}{\ln(2)}\left(\frac{16}{\pi}(2\pi)^{p-1}\right)^{\frac{1}{p+1}},\quad &\text{if } N=2\\
				\frac{4}{\pi^{\frac{1}{p+1}}},\quad&\text{if }N=1.
			\end{cases}
		\end{equation*}
	\end{teo}

	The paper is organized as follows. In Section \ref{sec: solucoes sao radialmente decescentes fora da bola de raio meio}, Theorem \ref{teorema solucoes decrescem fora da bola de raio meio} is proved using the moving planes method, following the classical approach introduced in \cite{Gidas1979}. Section \ref{sec: estimativas S alpha e S alpha rad} is devoted to the proof of Theorem \ref{teorema quebra de simetria}, based on a series of auxiliary estimates. In Section \ref{sec: estimativas norma L infinito}, a sharp bound on the growth of $|u_{\alpha,\mathrm{rad}}|_{\infty}$ as $\alpha \to \infty$ is established. Section \ref{sec: comportamento assitn solucao radial ground state} contains the proof of Theorem \ref{teorema comport assintotico sol radial} together with its corollaries, describing the asymptotic behavior of ground state radial solutions. Section \ref{sec: open problems} presents some open problems. The paper concludes in  Section \ref{sec: analise numerica}, with some numerical experiments based on the method developed in \cite{Tsitouras2019}.

	\section{Solutions are radially decreasing for $|x|>\frac{1}{2}$}\label{sec: solucoes sao radialmente decescentes fora da bola de raio meio}
	In this section, it is shown that smooth solutions to the problem \eqref{problema-principal} are radially decreasing outside the ball of radius $1/2$, thus proving Theorem~\ref{teorema solucoes decrescem fora da bola de raio meio}. This result follows directly from the moving planes method, as formulated in \cite[Theorem~2.1']{Gidas1979}. As a consequence, it also follows that the maximum of $u$ is attained within the ball $B_{\frac{1}{2}}(0)$.
	
	To maintain consistency with the notation used in \cite{Gidas1979} and to enhance readability, the following definition is introduced:
	\begin{equation*}
		F(x,v,P_1,\ldots,P_N) = \sum_{j=1}^{N} P_j + (4|x|(1-|x|))^\alpha |v|^{p-1}v
	\end{equation*}
	in $B\times [0,\infty) \times \mathbb{R}^N$. Thus solutions of (\ref{problema-principal}) satisfy 
	\begin{equation*}
		F\left(x,u(x),\frac{\partial^2 u}{\partial x_1 ^2}(x),\ldots,\frac{\partial^2 u}{\partial x_N^2}(x)\right)=0 .
	\end{equation*}
	To apply \cite[Theorem~2.1']{Gidas1979}, it suffices to verify that condition~(c) is satisfied, namely
	\begin{equation}\label{eq F(x_lambda,...) >= F(x,...)}
		F(x_\lambda,v,P_1,\ldots,P_N) \geq  F(x,v,P_1,\ldots,P_N),
	\end{equation}
	where $1/2 \leq \lambda \leq 1$, $x_\lambda = (2\lambda -x_1,x_2,\ldots,x_N)$ and for all $\lambda \leq x_1 \leq 1$.  Observe that \eqref{eq F(x_lambda,...) >= F(x,...)} is equivalent to
	\begin{equation}\label{NovaEqF}
		|x_\lambda|(1-|x_\lambda|) \geq |x|(1-|x|)
	\end{equation}
	where $1/2 \leq \lambda \leq 1$ and for all $\lambda \leq x_1 \leq 1$.

	\begin{lema}\label{lemma util para aplicar moving planes}
		Let $ 1/2 \leq\lambda\leq1$. Then,  for $x\in B$ and  $\lambda \leq x_1 \leq 1$, one has 
		\begin{equation*}
			|x_\lambda|(1-|x_\lambda|) - |x|(1-|x|)  =	\sqrt{4\lambda(\lambda-x_1)+ |x|^2 }  + 4\lambda(x_1-\lambda) -|x| \geq 0.
		\end{equation*}That is, \eqref{eq F(x_lambda,...) >= F(x,...)} holds.
	\end{lema}
	\begin{proof}Observe that for all $x \in B$ with $x_1 \geq \lambda \geq 1/2$, $|x_{\lambda}| = \sqrt{4\lambda(\lambda-x_1)+ |x|^2 }$, $|x-x_{\lambda}| = 2(x_1-\lambda)$ and
		\begin{align*}
			\sqrt{4\lambda(\lambda - x_1) + |x|^2} + 4\lambda(x_1 - \lambda) - |x|
			\geq 0
			\;\Longleftrightarrow\;
			4\lambda(x_1 -\lambda) \geq |x| - |x_{\lambda}|.
		\end{align*}
		Then, from the triangular inequality,
		\[
		|x| - |x_{\lambda}| \leq |x-x_{\lambda}| = 2 (x_1-\lambda) \leq 4\lambda(x_1-\lambda),
		\]
		since $x_1 \geq \lambda \geq 1/2$, which completes the proof.
	\end{proof}

	\begin{proof}[\textbf{Proof of Theorem \ref{teorema solucoes decrescem fora da bola de raio meio}}]
		By Lemma \ref{lemma util para aplicar moving planes}, using \cite[Theorem~2.1']{Gidas1979},  it follows that $\frac{\partial u}{\partial x_1}(x)<0$ for all $x \in B$ with $\frac{1}{2} \leq x_1 < 1$. Since the problem is rotationally invariant, the same arguments work for any radial direction, proving that $\frac{\partial u}{\partial r}<0$ in $B\backslash B_{\frac{1}{2}}(0)$. Consequently, the maximum of $u$ is attained at $B_{\frac{1}{2}}(0)$.
	\end{proof}
	
	\section{Estimates on $S_\alpha$ and $S_{\alpha,rad}$}\label{sec: estimativas S alpha e S alpha rad}

	In this section, the growths of $S_\alpha$ and $S_{\alpha,\mathrm{rad}}$ with respect to $\alpha$ are estimated. These estimates play a fundamental role in establishing the symmetry breaking result stated in Theorem~\ref{teorema quebra de simetria}, as well as in describing the asymptotic behavior of $u_{\alpha,\mathrm{rad}}$. The arguments used here are inspired by those developed in \cite[Sec.~2]{Mercuri2019} and \cite[Sec.~6]{SMETS2002}.

	\begin{lema}\label{lemma limitacao de S_alpha por cima em funcao de alpha}
		Let  $N \geq 1$ and $1< p < 2^*-1$, then
		\begin{equation}
			\limsup_{\alpha\rightarrow\infty} S_\alpha \alpha^{\frac{N}{2}-1-\frac{N}{p+1}} \leq e^{\frac{8}{p+1}}S.
		\end{equation}
	\end{lema}
	\begin{proof} Let $w \in H_0^1(B)\setminus\{0\}$ and set $x_{\frac{1}{2}} = \left(\frac{1}{2},0,\ldots,0\right)$. Define $\omega_\alpha$ by
		\[
		\omega_\alpha(x) = 
		\begin{cases}
			w\big(\sqrt{\alpha}(x - x_{\frac{1}{2}})\big), & x \in B_\alpha,\\
			0, & x \in B \setminus B_\alpha,
		\end{cases}
		\]
		where $B_\alpha = B\left(x_{\frac{1}{2}}, \frac{1}{\sqrt{\alpha}}\right)$. Then $\omega_\alpha \in H_0^1(B)$. Introducing the change of variables $y = \sqrt{\alpha}(x - x_{\frac{1}{2}})$, it follows that, for any $x \in B_\alpha$,
		\[
		\frac{1}{2} - \frac{1}{\sqrt{\alpha}} \leq \left| \frac{y}{\sqrt{\alpha}} + x_{\frac{1}{2}} \right| \leq \frac{1}{2} + \frac{1}{\sqrt{\alpha}}.
		\]
		By changing the variables, one obtains
		\begin{align*}
			S_\alpha \leq \frac{\int_B |D\omega_\alpha|^2(x)dx}{\left(\int_B V_\alpha(|x|)|\omega_\alpha|^{p+1}(x)dx\right)^{\frac{2}{p+1}}} \leq \frac{\alpha^{1+\frac{N}{p+1}-\frac{N}{2}}}{\left(\left(1-\frac{4}{\alpha}\right)^\alpha\right)^{\frac{2}{p+1}}}  \frac{\int_B |Dw|^2(x)dx}{\left(\int_B |w|^{p+1}(x)dx\right)^{\frac{2}{p+1}}}.
		\end{align*}
		Since $\lim_{\alpha \to \infty}\left(1 - \frac{4}{\alpha}\right)^\alpha = e^{-4}$, it follows that
		\[
		\limsup_{\alpha \to \infty} S_\alpha \, \alpha^{\frac{N}{2}-1-\frac{N}{p+1}} 
		\leq e^{\frac{8}{p+1}} S.
		\qedhere
		\]
	\end{proof}
	
	In the following, a sharp estimate for $S_{\alpha,\mathrm{rad}}$ is established.
	
	\begin{teo}\label{teorema estimativa sharp expoente de S alpha,rad em relacao a alpha}
		For $N\geq 1$ and $1< p <2^*_{\alpha}-1$ there exist constants $\mathcal{C}_1,\mathcal{C}_2>0$ such that, as $\alpha\rightarrow \infty$,
		\begin{equation*}
			\mathcal{C}_1 \alpha^{\frac{1}{p+1}}\leq S_{\alpha,rad}\leq \mathcal{C}_2 \alpha^{\frac{1}{p+1}}.
		\end{equation*}
	\end{teo}
	
	The proof begins with two preliminary lemmas.

	\begin{lema}\label{lemma estimativa dos quocientes com as funcos gama}
		Let $K_1,K_2\geq 0$. Then
		\begin{equation*}
			\frac{4^\alpha \Gamma(\alpha+K_1)\Gamma(\alpha+K_2)}{\Gamma(2\alpha+ K_1+K_2)} = \frac{\sqrt{\pi}}{2^{K_1+K_2-1}} \alpha^{-\frac{1}{2}}+ O(\alpha^{-\frac{3}{2}}),
		\end{equation*}
		as $\alpha\rightarrow \infty$.
	\end{lema}
	\begin{proof}
		By Legendre's relation \cite[Eq.~(3.11)]{Artin2015}, it holds that
		\begin{equation*}
			\Gamma(2\alpha+K_1+K_2)= \frac{2^{2\alpha+K_1+K_2-1}}{\sqrt{\pi}}\Gamma\left(\alpha+\frac{K_1+K_2}{2}\right)\Gamma\left(\alpha+\frac{K_1+K_2+1}{2}\right).
		\end{equation*}
		Consequently, 
		\begin{align*}
			\frac{4^\alpha \Gamma(\alpha+K_1)\Gamma(\alpha+K_2)}{\Gamma(2\alpha+ K_1+K_2)} =\frac{\sqrt{\pi}}{2^{K_1+K_2-1}}\left(\frac{\Gamma(\alpha+K_1)}{\Gamma\left(\alpha+\frac{K_1+K_2}{2}\right)}  \right)\left(\frac{\Gamma(\alpha+K_2)}{\Gamma\left(\alpha+\frac{K_1+K_2+1}{2}\right)}  \right).
		\end{align*}
		The conclusion then follows from the asymptotic estimate established in \cite{Tricomi1951}, which guarantees that
		\[
		\frac{\Gamma(\alpha +s)}{\Gamma(\alpha + r)} = \alpha^{s-r}\left[1 + \frac{(s-r)(s+r-1)}{2\alpha}+ O(\alpha^{-2})\right],
		\]
		as $\alpha \to +\infty$, provided $s$ and $r$ remain bounded.
	\end{proof}
	
	\begin{remk}\label{rmk: vendo a integral do nosso peso em funcao da funcao gamma}
		The Beta Function (also known as Euler's first integral) is the function
		\[B(x,y) \coloneq \int_{0}^{1}t^{x-1}(1-t)^{y-1}dt.\]
		By \cite[Eq.~(2.13)]{Artin2015}, it holds that
		\[B(x,y) = \frac{\Gamma(x)\Gamma(y)}{\Gamma(x+y)}.\]
		Since the integral of the weight reads as
		\begin{align*}
			\int_{B} V_\alpha(|x|)dx = \int_{B}(4|x|(1-|x|))^\alpha dx
			= |\mathbb{S}^{N-1}|4^{\alpha}\int_{0}^{1}r^{\alpha+N-1}(1-r)^{\alpha}dr,
		\end{align*}
		it can be written as
		\begin{equation*}
			\int_{B} V_\alpha(|x|)dx = |\mathbb{S}^{N-1}|4^{\alpha} B(\alpha+N,\alpha+1) =|\mathbb{S}^{N-1}| \frac{4^\alpha\Gamma(\alpha+N)\Gamma(\alpha+1)}{\Gamma(2\alpha + N+1)},
		\end{equation*}
		and Lemma \ref{lemma estimativa dos quocientes com as funcos gama} can be used to estimate this integral.
	\end{remk}

	The following lemma provides a version of 
	uniform convergence on compact sets for convolutions of continuous functions, analogous to \cite[Proposition 4.21]{2011_Brezis_BOOK}. 
	We emphasize that, differently from \cite{2011_Brezis_BOOK}, the regularizing function $\varphi_\alpha$ satisfies the \textit{approximated} normalization $\int_{-1}^{1}\varphi_\alpha(t)dt  = 1+ O(\alpha^{-1})$. Therefore, we include a proof for completeness. The argument follows, in part, the approach presented in \cite[p.~253]{2007_Lima_BOOK}.

	\begin{lema}\label{lemma aprox de weierstrass} Let   
		\begin{equation}\label{phialpha} \varphi_\alpha(x)= \begin{cases}
				\frac{\alpha^{\frac{1}{2}} }{\sqrt{\pi}}(1-x^2)^\alpha & \text{ if } |x|\leq 1,\\
				0 & \text{ if } |x|>1. 
		\end{cases}\end{equation}
		Then, for every continuous function $f:\mathbb{R}\rightarrow \mathbb{R}$, setting
		\[P_\alpha(x) = \int_{-1}^{1}f(x+t)\varphi_\alpha(t)dt,\]
		it holds that $\lim_{\alpha\rightarrow\infty} P_\alpha = f$ uniformly on any interval $[a,b]$. 
		
	\end{lema}
	\begin{proof}
		Since 
		\begin{equation*}
			\int_{-1}^{1}\varphi_\alpha(t)dt= \frac{2\alpha^{\frac{1}{2}}}{\sqrt{\pi}} \frac{4^\alpha\Gamma(\alpha+1)\Gamma(\alpha+1)}{\Gamma(2\alpha+2)},
		\end{equation*}
		using Lemma \ref{lemma estimativa dos quocientes com as funcos gama}, it follows that 
		\[g(\alpha) \coloneq \int_{-1}^{1}\varphi_\alpha(t)dt -1 = O(\alpha^{-1}), \qquad \text{as} \qquad \alpha  \to + \infty. \]
		Also, $\varphi_\alpha \to 0$ uniformly on the set $\{x; |x|\geq \sigma\}$, for any $0<\sigma<1$, as $\alpha \to + \infty$. Indeed, if $|x| \in [\sigma, 1]$
		\[0\leq \varphi_\alpha(x)= \frac{\alpha^{\frac{1}{2}}}{\sqrt{\pi}}(1-x^2)^\alpha \leq \frac{\alpha^{\frac{1}{2}}}{\sqrt{\pi}}(1-\sigma^2)^\alpha \to 0, \qquad \text{as} \qquad \alpha \to +\infty. \]
		On the other hand, since $f(x)$ is uniformly continuous on the interval $[a-1,b+1]$, given $\epsilon>0$, there exists $1>\delta>0$ such that, if $z,y \in [a-1,b+1]$ and $|z-y|<\delta$ then $|f(z)-f(y)|<\frac{\epsilon}{8}$. From the definition of $g(\alpha)$, it follows that
		\[f(x)= \int_{-1}^{1}f(x)\varphi_\alpha(t)dt - f(x)g(\alpha),\]
		and
		\begin{equation*}
			|f(x)-P_\alpha(x)|= \left|\int_{-1}^{1}[f(x+t)-f(x)]\varphi_\alpha(t)dt + f(x)g(\alpha)\right|.
		\end{equation*}
		Define $m= \sup_{x\in[a-1,b+1]}|f(x)|$.  Let $\alpha_1 > 0$ be such that $\alpha \geq \alpha_1$ implies $|g(\alpha)|<\min\{1,\frac{\epsilon}{4m}\}$
		and $|\varphi_\alpha(t)|<\frac{\epsilon}{8m}$ for any $|t|\geq \delta $. Then, for $x\in [a,b]$, 
		\begin{align*}
			|f(x)-P_\alpha(x)|&\leq \int_{-1}^{1}|f(x+t)-f(x)|\varphi_\alpha(t)dt + m|g(\alpha)|\\
			&=\int_{-1}^{-\delta}|f(x+t)-f(x)|\varphi_\alpha(t)dt + \int_{-\delta}^{\delta}|f(x+t)-f(x)|\varphi_\alpha(t)dt \\ &\hspace{1em}+\int_{\delta}^{1}|f(x+t)-f(x)|\varphi_\alpha(t)dt+ m|g(\alpha)|.
		\end{align*}
		Since
		\begin{equation*}
			\int_{-1}^{-\delta}|f(x+t)-f(x)|\varphi_\alpha(t)dt < 2m\frac{\epsilon}{8m} = \frac{\epsilon}{4}, 
		\end{equation*}
		\begin{equation*}
			\int_{-\delta}^{\delta}|f(x+t)-f(x)|\varphi_\alpha(t)dt < \frac{\epsilon}{8}\int_{-\delta}^{\delta}\varphi_\alpha(t)dt <\frac{\epsilon}{8}(1+g(\alpha))< \frac{\epsilon}{4},
		\end{equation*}
		\begin{equation*}
			\int_{\delta}^{1}|f(x+t)-f(x)|\varphi_\alpha(t)dt < 2m\frac{\epsilon}{8m} = \frac{\epsilon}{4} 
		\end{equation*}
		and
		\begin{equation*}
			m|g(\alpha)|< \frac{\epsilon}{4},
		\end{equation*}
		it follows that $|f(x)-P_\alpha(x)|<\epsilon$, for any $\alpha \ge \alpha_1$ and $x \in [a, b]$, and the lemma is proved.
	\end{proof}

	\begin{prop}\label{proposicao limitacao superior S alpha rad}
		For $N\geq 1$ and $1< p< 2_\alpha^*-1$ there exists a constant $\mathcal{C}_2>0$ such that  
		\[\limsup_{\alpha\rightarrow\infty} S_{\alpha,rad}\alpha^{-\frac{1}{p+1}} \leq \mathcal{C}_2. \]
	\end{prop}
	
	\begin{proof}

		Let $\phi$ be the first eigenfunction of $(-\Delta,H_0^1(B))$ such that $\int_B \phi^2 dx =1$ and $\phi$ is positive in $B$.  It is well known that $\phi$ is radial and radially decreasing.
		From the definition (\ref{minimizacao nas radiais}) of $S_{\alpha,rad}$, it follows that
		\begin{equation*}
			S_{\alpha,rad} \leq \frac{\int_B |D\phi|^2 dx}{\left(\int_B V_\alpha(|x|)|\phi|^{p+1}dx\right)^{2/p+1}}.
		\end{equation*}
		Since	
		\begin{align*}
			\int_B V_\alpha(|x|)|\phi|^{p+1}dx &= |\mathbb{S}^{N-1}|\int_{0}^{1}r^{N-1}(4r(1-r))^\alpha \phi(r)^{p+1}dr\\
			&=\frac{|\mathbb{S}^{N-1}|}{2}\int_{-1}^{1}\left(\frac{s+1}{2}\right)^{N-1}(1-s^2)^\alpha \phi^{p+1}\left(\frac{s+1}{2}\right)ds\\
			&=\frac{|\mathbb{S}^{N-1}|\sqrt{\pi}}{2\alpha^{\frac{1}{2}}}\int_{-1}^{1}\left(\frac{s+1}{2}\right)^{N-1}\varphi_\alpha(s) \phi^{p+1}\left(\frac{s+1}{2}\right)ds,
		\end{align*}
		with $\varphi_{\alpha}$ as in \eqref{phialpha}, one has that
		\begin{equation*}
			S_{\alpha,rad} \leq \frac{4^{\frac{1}{p+1}}\alpha^{\frac{1}{p+1}}}{\left(|\mathbb{S}^{N-1}|\sqrt{\pi}\right)^{\frac{2}{p+1}}} \frac{\int_B|D\phi|^2dx}{\left(\int_{-1}^{1}\left(\frac{s+1}{2}\right)^{N-1}\varphi_\alpha(s) \phi^{p+1}\left(\frac{s+1}{2}\right)ds\right)^{\frac{2}{p+1}}}.
		\end{equation*}
		Therefore, by Lemma \ref{lemma aprox de weierstrass},
		
		\begin{equation*}
			\limsup_{\alpha\rightarrow\infty} S_{\alpha,rad}\alpha^{-\frac{1}{p+1}} \leq \frac{4^{\frac{1}{p+1}}\int_B|D\phi|^2dx}{\left( \frac{|\mathbb{S}^{N-1}|\sqrt{\pi}}{2^{N-1}}\phi^{p+1}\left(\frac{1}{2}\right)   \right)^{\frac{2}{p+1}}}. \qedhere
		\end{equation*}
		
	\end{proof}
	
	\begin{prop}\label{proposicao limitacao inferior S_alpha,rad}
		Let $1< p <2_\alpha^*-1$. Then:
		
		i) For $N>2$,
		\begin{equation*}
			\left(\frac{|\mathbb{S}^{N-1}|^{p-1}4^N}{\pi}\right)^{\frac{1}{p+1}}\frac{(N-2)}{2^{N-2}}	\leq  \liminf_{\alpha\rightarrow \infty} S_{\alpha,rad}\alpha^{-\frac{1}{p+1}}.
		\end{equation*}
		ii) For $N=2$,
		\begin{equation*}
			\left(\frac{(2\pi)^{p-1}16}{\pi}\right)^{\frac{1}{p+1}}\frac{1}{\ln(2)} 	\leq  \liminf_{\alpha\rightarrow \infty} S_{\alpha,rad}\alpha^{-\frac{1}{p+1}}.
		\end{equation*}
		iii) For $N=1$,
		\begin{equation*}
			\frac{2}{\pi^{\frac{1}{p+1}}}	\leq  \liminf_{\alpha\rightarrow \infty} S_{\alpha,rad}\alpha^{-\frac{1}{p+1}}.
		\end{equation*}
	\end{prop}
	\begin{proof}
		Let $u\in H^1_{0,rad}(B)$. For $N\geq 3$, by \cite[Radial Lemma]{Ni1982},
		\begin{equation*}
			|u(x)|^{p+1}\leq \left(|\mathbb{S}^{N-1}|(N-2)\right)^{-\frac{p+1}{2}}\|Du\|_{L^2}^{p+1}|x|^{-\frac{(p+1)(N-2)}{2}}, \quad \text{for all } \quad x \in B\backslash\{0\}.
		\end{equation*}
		Thus
		\begin{align*}
			\int_B (4|x|(1-|x|))^\alpha |u|^{p+1}dx \leq&\left(|\mathbb{S}^{N-1}|(N-2)\right)^{-\frac{p+1}{2}}\|Du\|_{L^2}^{p+1}\\
			&\hspace{2em}\times \int_B |x|^{-\frac{(p+1)(N-2)}{2}}(4|x|(1-|x|))^\alpha dx,
		\end{align*}
		which implies that
		\begin{align*}
			\frac{|\mathbb{S}^{N-1}|(N-2)}{\left(\int_B |x|^{-\frac{(p+1)(N-2)}{2}}(4|x|(1-|x|))^\alpha dx \right)^{\frac{2}{p+1}}} \leq \frac{\int_{B} |Du|^2dx}{\left(\int_{B}(4|x|(1-|x|))^\alpha |u|^{p+1}dx\right)^{\frac{2}{p+1}}}.
		\end{align*}
		
		Notice that
		\begin{align*}
			\int_B |x|^{-\frac{(p+1)(N-2)}{2}}(4|x|(1-|x|))^\alpha dx &= |\mathbb{S}^{N-1}|\int_{0}^{1}r^{N-1-\frac{(p+1)(N-2)}{2}}(4r(1-r))^\alpha dr\\
			&=\frac{|\mathbb{S}^{N-1}|}{2}\int_{-1}^{1}\left(\frac{s+1}{2}\right)^{N-1-\frac{(p+1)(N-2)}{2}}(1-s^2)^\alpha ds\\
			&=\frac{|\mathbb{S}^{N-1}|\sqrt{\pi}}{2\alpha^{\frac{1}{2}}}\int_{-1}^{1}\left(\frac{s+1}{2}\right)^{N-1-\frac{(p+1)(N-2)}{2}}\varphi_\alpha(s) ds,
		\end{align*}
		with $\varphi_{\alpha}$ as in \eqref{phialpha}. Thus,
		\begin{align*}
			\frac{4^{\frac{1}{p+1}}\alpha^{\frac{1}{p+1}}|\mathbb{S}^{N-1}|(N-2)}{\left(|\mathbb{S}^{N-1}|\sqrt{\pi}\int_{-1}^{1}\left(\frac{s+1}{2}\right)^{N-1-\frac{(p+1)(N-2)}{2}}\varphi_\alpha(s)ds\right)^{\frac{2}{p+1}}} \leq \frac{\int_{B} |Du|^2dx}{\left(\int_{B}(4|x|(1-|x|))^\alpha |u|^{p+1}dx\right)^{\frac{2}{p+1}}}.
		\end{align*}
		and then
		\begin{align*}
			\frac{4^{\frac{1}{p+1}}\alpha^{\frac{1}{p+1}}|\mathbb{S}^{N-1}|(N-2)}{\left(|\mathbb{S}^{N-1}|\sqrt{\pi}\int_{-1}^{1}\left(\frac{s+1}{2}\right)^{N-1-\frac{(p+1)(N-2)}{2}}\varphi_\alpha(s)ds\right)^{\frac{2}{p+1}}} \leq S_{\alpha,rad}.
		\end{align*}
		Using Lemma \ref{lemma aprox de weierstrass}, it follows that
		\begin{align*}
			\left(\frac{|\mathbb{S}^{N-1}|^{p-1}4^N}{\pi}\right)^{\frac{1}{p+1}}\frac{(N-2)}{2^{N-2}}	\leq \liminf_{\alpha\rightarrow \infty} S_{\alpha,rad}\alpha^{-\frac{1}{p+1}}.
		\end{align*}
		
		For $N=2$, using the radial lemma \cite[Lemma 2.5]{Bonheure2008},
		\begin{equation*}
			|u(x)|^{p+1}\leq (2\pi)^{-\frac{p+1}{2}}|\ln|x||^{\frac{p+1}{2}}\|Du\|_{L^2}^{p+1}, \quad \text{for all } \quad x \in B\backslash\{0\}.
		\end{equation*}
		Thus, proceeding as in the case $N\geq 3$, it follows that
		\begin{align*}
			\frac{4^{\frac{1}{p+1}}\alpha^{\frac{1}{p+1}}(2\pi)}{\left(2\pi \sqrt{\pi} \int_{-1}^{1}\left(\frac{s+1}{2}\right)\left| \ln\left(\frac{s+1}{2}\right)\right|^{\frac{p+1}{2}} \varphi_\alpha(s)ds\right)^\frac{2}{p+1}}\leq S_{\alpha,rad}.
		\end{align*}
		Therefore, by Lemma \ref{lemma aprox de weierstrass},
		\begin{align*}
			\left(\frac{16(2\pi)^{p-1}}{\pi}\right)^{\frac{1}{p+1}}\frac{1}{\ln(2)} \leq	 \liminf_{\alpha\rightarrow \infty} S_{\alpha,rad}\alpha^{-\frac{1}{p+1}}.
		\end{align*}
		Finally, for $N=1$, using Cauchy-Schwarz inequality,
		\begin{equation*}
			|u(x)|= \left|\int_{x}^{sign(x)} u'(t)dt\right|\leq \|u'\|_{L^2} (1- |x|)^{1/2} \quad \text{for all} \quad x\in[-1,1].
		\end{equation*}
		Then, applying the same reasoning of the previous cases, the desired inequality follows.
	\end{proof}

	\begin{proof}[\textbf{Proof of Theorem \ref{teorema estimativa sharp expoente de S alpha,rad em relacao a alpha}}]
		It is a straightforward consequence of Propositions \ref{proposicao limitacao superior S alpha rad} and \ref{proposicao limitacao inferior S_alpha,rad}.
	\end{proof}
	
	\begin{proof}[\textbf{Proof of Theorem \ref{teorema quebra de simetria}}]
		
		Since $p>\frac{N}{N-2}$ implies $\frac{1}{p+1}>1+ \frac{N}{p+1}-\frac{N}{2}$, it follows from Lemma \ref{lemma limitacao de S_alpha por cima em funcao de alpha} and Theorem \ref{teorema estimativa sharp expoente de S alpha,rad em relacao a alpha} that $S_\alpha<S_{\alpha,rad}$ for all $\alpha$ large.
		On the other hand, from \eqref{eq util relacao entre S _alpha,rad e C_alpha,rad},
		\[ S_\alpha \leq S_{\alpha,rad} \text{ if, and only if, } C_\alpha\leq C_{\alpha,rad}.\]
		Therefore, $C_\alpha < C_{\alpha,rad}$ and ground state solutions are not radial for all $\alpha$ large.
	\end{proof}

	\section{$L^\infty$-norm estimates for $u_{\alpha,rad}$}\label{sec: estimativas norma L infinito}
	In this section, the growth of $\|u_{\alpha,\mathrm{rad}}\|_\infty$ with respect to $\alpha$ is estimated. This analysis is useful for understanding the asymptotic behavior of $u_{\alpha,\mathrm{rad}}$ through the study of the normalized function $u_{\alpha,\mathrm{rad}}/\|u_{\alpha,\mathrm{rad}}\|_\infty$. Since $u_{\alpha,\mathrm{rad}}$ is radial and radially decreasing, it follows that
	\[
	u_{\alpha,\mathrm{rad}}(0)=\|u_{\alpha,\mathrm{rad}}\|_\infty .
	\]

	\begin{prop}\label{prop limitacao por baixo norma do sup sol rad}
		Let $N\geq 1$ and $1<p<2_\alpha^*-1$. There exists a constant $\mathcal{C}_3>0$, depending only on $N$ and $p$, such that 
		\begin{equation*}
			\mathcal{C}_3 \alpha^{\frac{1}{2(p-1)}} \leq  \| u_{\alpha,rad}\|_\infty .
		\end{equation*}
		as $\alpha\rightarrow \infty$.
	\end{prop}
	\begin{proof}
		Notice that
		\begin{equation*}
			\| u_{\alpha,rad}\|^{p+1}_\infty \geq \frac{\int_B V_\alpha(|x|)u_{\alpha,rad}^{p+1}dx}{\int_B V_\alpha(|x|)dx}.
		\end{equation*}

		Since $u_{\alpha, rad}$ is a \emph{ground state} radial solution,
		\[\int_B |Du_{\alpha,rad}|^2 dx = \int_B V_\alpha(|x|)u_{\alpha,rad}^{p+1}dx = S_{\alpha,rad}^{\frac{p+1}{p-1}}.\]
		
		Using Theorem \ref{teorema estimativa sharp expoente de S alpha,rad em relacao a alpha}  one obtains
		
		\begin{equation}\label{usando teo limt Salpha p limitar norma de u L infity}
			\mathcal{C}_1\alpha^{\frac{1}{p-1}}\leq \int_B |Du_{\alpha,rad}|^2 dx = \int_B V_\alpha(|x|)u_{\alpha,rad}^{p+1}dx\leq \mathcal{C}_2\alpha^{\frac{1}{p-1}}.
		\end{equation}
		
		Thus, as in Remark \ref{rmk: vendo a integral do nosso peso em funcao da funcao gamma}, using Lemma \ref{lemma estimativa dos quocientes com as funcos gama}, 
		\begin{align*}
			\| u_{\alpha,rad}\|_\infty^{p+1}\geq \frac{\mathcal{C}_1 \alpha^{\frac{1}{p-1}}}{\int_{B} V_\alpha(|x|) dx} = \frac{\mathcal{C}_1 \alpha^{\frac{1}{p-1}}}{\alpha^{-\frac{1}{2}}\left(\frac{\sqrt{\pi}|\mathbb{S}^{N-1}|}{2^N}+ O (\alpha^{-1})\right)}		
		\end{align*}
		and for a constant $\mathcal{C}>0$, one has that
		\begin{equation*}
			\| u_{\alpha,rad}\|_\infty \geq \left(\mathcal{C} \alpha^{\frac{p+1}{2(p-1)}} \right)^{\frac{1}{p+1}} = \mathcal{C}^{\frac{1}{p+1}} \alpha^{\frac{1}{2(p-1)}}. \qedhere
		\end{equation*}
	\end{proof}
	
	\begin{prop}\label{prop limitacao por cima norma do sup sol rad}
		Let $N\geq 1$ and $1<p<2_\alpha^*-1$. There exists a constant $\mathcal{C}_4>0$, depending only on $N$ and $p$, such that 
		\begin{equation}\label{desejoestimativacima}
			\| u_{\alpha,rad}\|_\infty \leq \mathcal{C}_4 \alpha^{\frac{1}{2(p-1)}}
		\end{equation}
		as $\alpha\rightarrow \infty$.
	\end{prop}	
	\begin{proof} 
		The proof is inspired by \cite[Lemma~3.2]{LeitedaSilva2021}. Let $M>2$ be such that $p<M/(M-2)$. 
		
		\textbf{Case $N >2$.} Set 
		\begin{equation}\label{mudanca boa limt por baix norm L infity}
			v_\alpha(t) = \frac{1}{\alpha^{\frac{1}{2(p-1)}}} u_{\alpha,rad}\left(t^{\frac{M-2}{N-2}}\right),\quad t\in [0,1].
		\end{equation}
		
		One has that $v_\alpha$ satisfies 
		\begin{equation*}
			\begin{cases}
				\frac{-1}{t^{M-1}}(t^{M-1}v'_\alpha)' = \alpha^{\frac{1}{2}}\left(\frac{M-2}{N-2}\right)^2t^{2\left(\frac{M-2}{N-2}-1\right)}\left(4t^{\frac{M-2}{N-2}}\left(1- t^{\frac{M-2}{N-2}}\right)\right)^\alpha v_\alpha^p\quad \text{in }(0,1),\\
				v_\alpha(1) = v'_\alpha(0)=0.
			\end{cases}
		\end{equation*}
		Thus
		\begin{equation}\label{express v'}
			v'_\alpha(t) =  \frac{-1}{t^{M-1}}\alpha^{\frac{1}{2}}\left(\frac{M-2}{N-2}\right)^2\int_{0}^{t}s^{M-1}s^{2\left(\frac{M-2}{N-2}-1\right)}\left(4s^{\frac{M-2}{N-2}}\left(1- s^{\frac{M-2}{N-2}}\right)\right)^\alpha v_\alpha^p(s) ds.
		\end{equation}
		Using \eqref{usando teo limt Salpha p limitar norma de u L infity} and \eqref{mudanca boa limt por baix norm L infity}, there exists a positive constant $\hat C$ such that 
		\begin{align}\label{estimate v'}
			\int_{0}^{1}|v'_\alpha|^2(t)t^{M-1} dt &= \left(\frac{M-2}{N-2}\right)^2\frac{1}{\alpha^{\frac{1}{p-1}}}\int_{0}^{1}|u'_{\alpha,rad}|^2\left(t^{\frac{M-2}{N-2}}\right)t^{2\left(\frac{M-2}{N-2}-1\right)}t^{M-1}dt\nonumber \\
			&= \left(\frac{M-2}{N-2}\right)\frac{1}{\alpha^{\frac{1}{p-1}}}\int_{0}^{1}|u'_{\alpha,rad}|^2(r)r^{N-1}dr \nonumber\\
			&\leq \hat{C}\left(\frac{M-2}{N-2}\right).
		\end{align}
		Suppose, by contradiction, that the upper bound \eqref{desejoestimativacima} does not hold. Then there exists a sequence $(\alpha_n)$, with $\alpha_n \rightarrow \infty$, such that  $\|v_{\alpha_n}\|_\infty \rightarrow \infty$ as $n\rightarrow\infty$. Define 
		\begin{equation}\label{mudanca blowup}
			\overline{v}_{\alpha_n}(s) = \frac{1}{\|v_{\alpha_n}\|_\infty}v_{\alpha_n}\left(\frac{s}{\|v_{\alpha_n}\|_\infty^{p-1}}\right), \quad \text{for } s\in[0,\|v_{\alpha_n}\|_\infty^{p-1}].
		\end{equation}
		Thus, $\|\overline{v}_{\alpha_n}\|_\infty =1$ and by \eqref{estimate v'}
		\begin{align}\label{estimate v bar}
			\int_{0}^{\|v_{\alpha_n}\|_\infty^{p-1}}|\overline{v}'_{\alpha_n}|^2(s)s^{M-1}ds &=\|v_{\alpha_n}\|_\infty^{-2p}\int_{0}^{\|v_{\alpha_n}\|_\infty^{p-1}}|v'_{\alpha_n}|^2\left(\frac{s}{\|v_{\alpha_n}\|_\infty^{p-1}}\right)s^{M-1}ds \nonumber\\
			&= \|v_{\alpha_n}\|_\infty^{(M-2)p-M}\int_{0}^{1}|v'_{\alpha_n}|^2(t)t^{M-1}dt\nonumber\\
			&\leq \hat{C}\left(\frac{M-2}{N-2}\right)\|v_{\alpha_n}\|_\infty^{(M-2)p-M}.
		\end{align}
		Since, by the choice of $M$, $(M-2)p-M <0$, the last term in \eqref{estimate v bar} goes to 0 as $\alpha \to \infty$.
		Therefore one concludes that $\overline{v}_{\alpha_n}\rightarrow 0$ in $D^{1,2}_{M}(0,\infty)$, as in \cite[Lemma~3.2]{LeitedaSilva2021}, where
		\begin{equation*}\label{espaco D Mepsilon}
			D^{1,2}_{M} (0, \infty)\coloneq \left\{\begin{split}
				w\in L^{2^*_{M}}_{M}(0,\infty)
			\end{split}; \begin{split}
				&w \text{ has first order derivative and}\\
				\|&w\|^2_{D^{1,2}_{M}} \coloneq \int_{0}^{\infty}|w'(t)|^2 t^{M-1}dt <\infty 
			\end{split} \right\}
		\end{equation*}
		and 
		\[L^{2^*_{M}}_{M}(0,\infty) \coloneq \left\{w:[0,\infty)\rightarrow \mathbb{R}\text{ measurable }; \int_{0}^{\infty}|w(t)|^{\frac{2M}{M -2}} t^{M-1} dt < \infty \right\}. \]
		
		The next step is to show that $|\overline{v}'_{\alpha_n}(s)|$ is uniformly bounded from above. Notice that
		\begin{equation}\label{novaderivada}
			|\overline{v}'_{\alpha_n}|(s) = \|v_{\alpha_n}\|_\infty^{-p}\left|v'_{\alpha_n}\left(\frac{s}{\|v_{\alpha_n}\|_\infty^{p-1}}\right)\right|,
		\end{equation}
		and by \eqref{express v'}, for any $s\le \|v_{\alpha_n}\|_\infty^{p-1}$, one has
		\begin{align}\label{novaderivada2}
			\left|v'_{\alpha_n}\left(\frac{s}{\|v_{\alpha_n}\|_\infty^{p-1}}\right)\right| &= \alpha_n^{\frac{1}{2}}\left(\frac{M-2}{N-2}\right)^2\left(\frac{1}{\frac{s}{\|v_{\alpha_n}\|_\infty^{p-1}}}\right)^{M-1}\hspace{12em} \nonumber\\
			& \hspace{1cm} \times \int_{0}^{\frac{s}{\|v_{\alpha_n}\|_\infty^{p-1}}}t^{M-1}t^{2\left(\frac{M-2}{N-2}-1\right)}\left(4t^{\frac{M-2}{N-2}}\left(1- t^{\frac{M-2}{N-2}}\right)\right)^{\alpha_n} v_{\alpha_n}^p(t) dt \nonumber \\
			&\leq  \alpha_n^{\frac{1}{2}}\left(\frac{M-2}{N-2}\right)^2\|v_{\alpha_n}\|_\infty^p\int_{0}^{1}t^{2\left(\frac{M-2}{N-2}-1\right)}\left(4t^{\frac{M-2}{N-2}}\left(1- t^{\frac{M-2}{N-2}}\right)\right)^{\alpha_n}  dt \nonumber\\
			&= \left(\frac{M-2}{N-2}\right)\|v_{\alpha_n}\|_\infty^p\alpha_n^{\frac{1}{2}}\int_{0}^{1}4^{\alpha_n} r^{\alpha_n+\frac{2-N}{M-2}+1}(1-r)^{\alpha_n} dr\hspace{5.9em} \nonumber\\
			&=\left(\frac{M-2}{N-2}\right)\|v_{\alpha_n}\|_\infty^p\alpha_n^{\frac{1}{2}} 4^{\alpha_n} \frac{\Gamma(\alpha_n +\frac{2-N}{M-2}+2)\Gamma(\alpha_n+1)}{\Gamma(\alpha_n+\frac{2-N}{M-2}+3)}\hspace{5.9em}	\nonumber \\
			&=\left(\frac{M-2}{N-2}\right)\|v_{\alpha_n}\|_\infty^p \frac{\sqrt{\pi}}{2^{\frac{N-2}{M-2}+2}}\left(1 + O(\alpha_n^{-1})\right), \hspace{9.2em}
		\end{align}
		where, in the last equality, Lemma \ref{lemma estimativa dos quocientes com as funcos gama} is used. Hence, from \eqref{novaderivada} and \eqref{novaderivada2}
		\begin{equation}\label{v' limitada}
			|\overline{v}'_{\alpha_n}|(s) < \left(\frac{M-2}{N-2}\right) \frac{\sqrt{\pi}}{2^{\frac{2-N}{M-2}+1}},
		\end{equation}
		as $n\rightarrow\infty$, for any $s$.
		
		Recall that, as a consequence of \eqref{estimate v bar}, $\overline{v}_{\alpha_n}\rightarrow 0$ in $D^{1,2}_{M}(0, \infty)$. By Sobolev embeddings, the inclusion $D^{1,2}_{M}(0,\infty) \subset L^{2^*_{M}}_{M}(0,\infty)$ is continuous, thus, taking a subsequence if necessary, $\overline{v}_{\alpha_n}\rightarrow 0$ a.e on $[0,\infty)$. But $\|\overline{v}_{\alpha_n}\|_\infty =1$ and by \eqref{v' limitada} it follows that $|\overline{v}'_{\alpha_n}|$ is uniformly bounded. Thus, by the Arzelà-Ascoli Theorem, see \cite[Theorem~4.44]{Folland1999}, there exists a subsequence of $\overline{v}_{\alpha_n}$ whose restriction to $[0,1]$ converges uniformly to a function $v\in C([0,1], \mathbb{R})$. Therefore $v$ is a continuous function such that $\|v\|_\infty =1$, contradicting the fact that $\overline{v}_{\alpha_n}\rightarrow 0$ a.e on $[0,\infty)$.

		\textbf{Case $N =2$.} Define
		\begin{equation*}
			v_\alpha(t) = \frac{1}{\alpha^{\frac{1}{2(p-1)}}} u_{\alpha,rad}\left(t\right),\quad t\in [0,1].
		\end{equation*}
		By Gagliardo--Nirenberg--Sobolev inequality \cite[Theorem~1, Section~5.6]{2010_Evans_BOOK}, there exists a positive constant $\mathcal{K}$ such that
		\begin{equation}\label{usando GaglNirSob}
			\|v_\alpha\|_{L^2} \leq \mathcal{K}\|Dv_\alpha\|_{L^1}.
		\end{equation}
		By Hölder's inequality, it follows that
		\begin{align*}
			\int_B |Dv_\alpha|dx = 2\pi\int_{0}^{1}|v'_\alpha|(t)t dt \leq \sqrt{2}\pi \left(\int_{0}^{1}|v'_\alpha|^2(t)tdt\right)^{\frac{1}{2}} 
		\end{align*}
		and, using the definition of $v_\alpha$ and \eqref{usando teo limt Salpha p limitar norma de u L infity}, one has
		\begin{align}\label{eq limitando norma L2 grad, N=2}
			\int_{0}^{1}|v'_\alpha|^2(t)tdt = \frac{1}{\alpha^{\frac{1}{p-1}}}\int_{0}^{1}|u'_{\alpha,rad}|^2(t) t dt \leq \frac{\mathcal{C}_2}{2\pi}.
		\end{align}
		Thus
		\begin{equation}\label{eq limitando norma L1 do grad, N=2}
			\int_B |Dv_\alpha|dx  \leq \pi^{\frac{1}{2}}\mathcal{C}_2^{\frac{1}{2}}
		\end{equation}
		and by \eqref{usando GaglNirSob} and \eqref{eq limitando norma L1 do grad, N=2} it follows
		\begin{equation}\label{limitacao por cima norma L2 independende de alpha, N=2}
			\|v_\alpha\|_{L^2} \leq \mathcal{K}\pi^{\frac{1}{2}}\mathcal{C}_2^{\frac{1}{2}}.
		\end{equation}
		Moreover, from \eqref{eq limitando norma L2 grad, N=2} and \eqref{limitacao por cima norma L2 independende de alpha, N=2} one obtains
		\begin{equation}\label{limitando norma W1,2 por cima independente de alpha, N=2}
			\|v_\alpha\|_{W^{1,2}(B)} \leq \left((\mathcal{K}^2\pi+1)\mathcal{C}_2\right)^{\frac{1}{2}}.
		\end{equation}
		Notice that the upper bound of \eqref{limitando norma W1,2 por cima independente de alpha, N=2} is \emph{independent} of $\alpha$. By Rellich--Kondrachov Theorem \cite[Theorem~9.16]{2011_Brezis_BOOK} one has that $W^{1,2}(B)\subset L^{q}(B)$ for all $q\in [1,\infty)$, in particular $W^{1,2}(B)\subset L^{p+1}(B)$. Therefore by \eqref{limitando norma W1,2 por cima independente de alpha, N=2} there exists a constant $\mathcal{C}>0$ independent of $\alpha$ such that
		\begin{equation}\label{eq limitacao v alpha norma L p+1, N=2}
			\|v_\alpha\|_{L^{p+1}} \leq \mathcal{C}.
		\end{equation}
		Since $v_\alpha$ satisfies
		\begin{equation*}
			\begin{cases}
				-\frac{1}{t}(tv'_\alpha)' = \alpha^{\frac{1}{2}}(4t(1-t))^\alpha v_\alpha^p,\quad \text{in }(0,1)\\
				v_\alpha(1) = v'_\alpha(0)=0,
			\end{cases}
		\end{equation*}
		by Hölder's inequality,
		\begin{align*}
			|v'_\alpha(t)|^{p+1} =& \left(\frac{1}{t}\int_{0}^{t}\alpha^{\frac{1}{2}}s(4s(1-s))^\alpha v^p_\alpha(s)ds  \right)^{p+1}\\
			&\leq \left(\frac{1}{t}\left(\int_{0}^{t}\alpha^{\frac{1}{2}}s(4s(1-s))^\alpha v_\alpha^{p+1}ds \right)^\frac{p}{p+1}\left(\int_{0}^{t}\alpha^{\frac{1}{2}}s(4s(1-s))^\alpha ds\right)^{\frac{1}{p+1}} \right)^{p+1}.
		\end{align*}
		Thus
		\begin{align}\label{eq util para limitar norma L p+1 v' alpha, N=2}
			|v'_\alpha(t)|^{p+1}t &\leq  \frac{1}{t^{p}}\left(\int_{0}^{t}\alpha^{\frac{1}{2}}s(4s(1-s))^\alpha v_\alpha^{p+1}ds \right)^p \int_{0}^{t}\alpha^{\frac{1}{2}}s(4s(1-s))^\alpha ds\nonumber\\
			&\leq \left(\int_{0}^{t}\alpha^{\frac{1}{2}}s(4s(1-s))^\alpha v_\alpha^{p+1}ds \right)^p\alpha^{\frac{1}{2}}4^\alpha\int_{0}^{t}s^{\alpha+1-p}(1-s)^\alpha ds\nonumber\\
			&\leq \left(\int_{0}^{1}\alpha^{\frac{1}{2}}s(4s(1-s))^\alpha v_\alpha^{p+1}ds \right)^p\alpha^{\frac{1}{2}}4^\alpha\int_{0}^{1}s^{\alpha+1-p}(1-s)^\alpha ds.
		\end{align}
		Using the definition of $v_\alpha$ and \eqref{usando teo limt Salpha p limitar norma de u L infity} it follows that
		\begin{equation*}
			\int_{0}^{1}\alpha^{\frac{1}{2}}s(4s(1-s))^\alpha v_\alpha^{p+1}ds = \alpha^{-\frac{1}{p-1}}\int_{0}^{1}s(4s(1-s))^\alpha u_{\alpha,rad}^{p+1}\;ds  \leq \frac{\mathcal{C}_2}{2\pi}.
		\end{equation*}
		Using Lemma \ref{lemma estimativa dos quocientes com as funcos gama} one obtains
		\begin{equation*}
			\alpha^{\frac{1}{2}}4^\alpha\int_{0}^{1}s^{\alpha+1-p}(1-s)^\alpha ds = \frac{\sqrt{\pi}}{2^{2-p}} + O(\alpha^{-1}).
		\end{equation*}
		Therefore, by \eqref{eq util para limitar norma L p+1 v' alpha, N=2}, one has
		\begin{equation}\label{eq limitacao v' alpha norma L p+1, N=2}
			\int_{0}^{1}|v'_\alpha(t)|^{p+1}t dt \leq \left(\frac{\mathcal{C}_2}{2\pi}\right)^p\frac{\sqrt{\pi}}{2^{1-p}}.
		\end{equation}
		Notice that the constant on the right hand side of \eqref{eq limitacao v' alpha norma L p+1, N=2} is \emph{independent} of $\alpha$. From \eqref{eq limitacao v alpha norma L p+1, N=2} and \eqref{eq limitacao v' alpha norma L p+1, N=2} it follows that there exists a constant $K>0$ independent of $\alpha$ such that
		\begin{equation}\label{eq limitacao norma W 1 p+1, N=2}
			\|v_\alpha\|_{W^{1,p+1}(B)} \leq K.
		\end{equation}
		Since $p+1>2$, by Morrey's inequality \cite[Theorem~9.12]{2011_Brezis_BOOK} and \eqref{eq limitacao norma W 1 p+1, N=2} there exist a constant $\mathcal{C}_4>0$ independent of $\alpha$ and such that
		\begin{equation*}
			\|v_\alpha\|_{\infty}\leq \mathcal{C}_4.
		\end{equation*}
		This proves the case $N=2$.
		
		\textbf{Case $N =1$.} As in the case $N=2$, set
		\begin{equation*}
			v_\alpha(t) = \frac{1}{\alpha^{\frac{1}{2(p-1)}}} u_{\alpha,rad}\left(t\right),\quad t\in [-1,1].
		\end{equation*}
		By \eqref{usando teo limt Salpha p limitar norma de u L infity} it follows that
		\begin{equation*}
			\|v'_\alpha\|_{L^2}\leq \mathcal{C}_2^{\frac{1}{2}}.
		\end{equation*}
		Then, Poincaré's inequality \cite[Proposition~8.13]{2011_Brezis_BOOK} and Sobolev embedding \cite[Theorem~8.8]{2011_Brezis_BOOK} imply that there exists a constant $\mathcal{C}_4$ independent of $\alpha$ such that
		\begin{equation*}
			\|v_\alpha\|_\infty \leq \mathcal{C}_4.
		\end{equation*}
		This ends the proof.	
	\end{proof}

	\section{Asymptotic profile of ground state radial solutions}\label{sec: comportamento assitn solucao radial ground state}

	Set
	\begin{equation}\label{mudanca de var util usada na aprx da derivada em 1}
		y_\alpha(x) := \frac{u_{\alpha,rad}(x)}{\left(\frac{\alpha^{\frac{1}{2}}}{\sqrt{\pi}}\right)^{\frac{1}{p-1}}}.
	\end{equation}
	Thus 
	\begin{equation}\label{eq: equacao q y alpha satisfaz}
		-\Delta y_\alpha(x)=\left(\frac{\alpha^{\frac{1}{2}}}{\sqrt{\pi}}\right)(4|x|(1-|x|))^\alpha y_\alpha^p(x), \quad \forall \, x \in B.
	\end{equation}
	
	\begin{lema}\label{lemma util que usa a aproximacao de weierstrass}
		Let $N\geq 1$ and $1<p<2_\alpha^*-1$ and $y_{\alpha}$ be as in \eqref{mudanca de var util usada na aprx da derivada em 1}. Then
		\begin{equation}
			\lim_{\alpha\rightarrow\infty}\!\! \frac{\alpha^{\frac{1}{2}}}{\sqrt{\pi}}\left|\int_{-1}^{1}(1-t^2)^\alpha(t+1)^{N-1}y_\alpha^p\left(\frac{t+1}{2}\right)dt - y_\alpha^p\left(\frac{1}{2}\right)\int_{-1}^{1}(1-t^2)^\alpha(1-t)dt \right| =0.
		\end{equation}
	\end{lema}
	\begin{proof}
		First notice that from Propositions \ref{prop limitacao por cima norma do sup sol rad} and \ref{prop limitacao por baixo norma do sup sol rad}, for $\alpha$ sufficiently  large, $\sqrt{\pi}^{\frac{1}{p-1}}\mathcal{C}_3\leq \|y_\alpha\|_\infty \leq \sqrt{\pi}^{\frac{1}{p-1}}\mathcal{C}_4$. Using equation \eqref{eq: equacao q y alpha satisfaz} it follows that
		\[y'_\alpha(r) = -\frac{1}{r^{N-1}}\int_{0}^{r}\left(\frac{\alpha^{\frac{1}{2}}}{\sqrt{\pi}}\right)(4s(1-s))^\alpha s^{N-1}y_\alpha^p(s)ds. \]
		Combining this identity with Lemma \ref{lemma estimativa dos quocientes com as funcos gama}, one concludes that 
		\[|y'_\alpha(r)|\leq \frac{\sqrt{\pi}^{\frac{p}{p-1}}\mathcal{C}_4^p}{2}+ O(\alpha^{-1}),\]
		and so $\|y'_\alpha\|_\infty \leq \sqrt{\pi}^{\frac{p}{p-1}}\mathcal{C}_4^p$ for $\alpha$ sufficiently large. 
		
		Furthermore, 
		\[-y'_\alpha(1) = \left(\frac{\alpha^{\frac{1}{2}}}{\sqrt{\pi}}\right)\int_{0}^{1}(4s(1-s))^\alpha s^{N-1}y_\alpha^p(s)ds, \]
		and making the change of variables $s = \frac{t+1}{2}$, 
		\begin{equation}\label{aproximacao util para derivada no 1}
			-y'_\alpha(1) = \left(\frac{\alpha^{\frac{1}{2}}}{2^N\sqrt{\pi}}\right)\int_{-1}^{1}(1-t^2)^\alpha\underbrace{(t+1)^{N-1}y^p_\alpha\left(\frac{t+1}{2}\right)}_{g_\alpha(t)} dt.
		\end{equation}
		Define $\tilde{g}_\alpha(t)\coloneq g_\alpha(t)-g_\alpha(0)-t(g_\alpha(1)-g_\alpha(0))$, so that $\tilde{g}_\alpha(0)=\tilde{g}_\alpha(1)=0$. Moreover,
		\[\tilde{g}_\alpha(t) = (t+1)^{N-1}y_\alpha^p\left(\frac{t+1}{2}\right)-y_\alpha^p\left(\frac{1}{2}\right)(1-t).\]
		Define
		\[P_\alpha(x)= \int_{-1}^{1}\tilde{g}_\alpha(x+t)\varphi_\alpha(t)dt\]
		and $\varphi_\alpha$ as in Lemma \ref{lemma aprox de weierstrass}. Notice that, by the boundedness of $\|y_\alpha\|_\infty,\|y'_\alpha\|_\infty$, the set $\{\tilde{g}_\alpha,\alpha\geq \alpha_0\}$ is equicontinuous, where $\alpha_0$ is taken sufficiently large. Then the equicontinuity property allows one to adapt the proof of Lemma \ref{lemma aprox de weierstrass} and conclude that 
		\[\lim_{\alpha\rightarrow\infty}|P_\alpha(0)-\tilde{g}_\alpha(0)|=0 \]
		and this ends the proof.
	\end{proof}
	
	\begin{remk}
		We use the notation $A_\alpha \approx B_{\alpha}$ to indicate that $|A_\alpha - B_{\alpha}|\rightarrow 0$ as $\alpha \rightarrow \infty$. Observe that the quantities $A_{\alpha}$ and $B_{\alpha}$ in the corollary below are bounded.
	\end{remk}    
	\begin{corl}\label{corolario util na aprox da derivada em 1}
		For $N\geq 1$, $1<p<2_\alpha^*-1$ and $\alpha$ large enough,
		\begin{equation*}
			-y'_\alpha(1) \approx \left(\frac{\alpha^{\frac{1}{2}}}{2^N\sqrt{\pi}} \right) y_\alpha^p\left(\frac{1}{2}\right)\int_{-1}^{1}(1-t^2)^\alpha(1-t)dt.
		\end{equation*}
	\end{corl}

	\begin{proof}
		It follows from \eqref{aproximacao util para derivada no 1} and Lemma \ref{lemma util que usa a aproximacao de weierstrass}.
	\end{proof}

	\begin{lema}\label{lema da mudanca magica}
		Let $f\in C^2(A, [0,1])$ be such that $|f(t)f'(t)|>0$ for all $t\in A \subset \mathbb{R}$. For $N>1$, if $u$ satisfies 
		\[u''(t) + \frac{N-1}{t}u'(t)= -V(t)u^p(t)\]
		then
		$w(t) = u(f(t))$ satisfies 
		\begin{equation*}
			\left(\frac{f(t)^{N-1}}{f'(t)}w' \right)'= -f(t)^{N-1}f'(t)V(f(t)) w^p(t).
		\end{equation*}
	\end{lema}
	
	\begin{proof}
		Indeed, for a $\sigma$ to be determined, one has
		\begin{align*}
			w''(t) + \sigma w'(t) &=u''(f(t))(f'(t))^2 + u'(f(t))(\sigma f'(t)+f''(t))\\
			&= (f'(t))^2\left(u''(f(t)) + u'(f(t))(\sigma f'(t)+f''(t))\frac{1}{(f'(t))^2} \right)\\	
			&=(f'(t))^2\left(u''(f(t)) + \frac{1}{f(t)}u'(f(t))(\sigma f'(t)+f''(t))\frac{f(t)}{(f'(t))^2} \right) .	 
		\end{align*}
		Take $\sigma$ such that
		\[(\sigma f'(t)+f''(t))\frac{f(t)}{(f'(t))^2} = N-1,\]
		thus
		\begin{equation*}
			\sigma = \frac{f'(t)}{f(t)}(N-1) - \frac{f''(t)}{f'(t)}.
		\end{equation*}
		Therefore, 
		\begin{equation*}
			w''(t) + \left(\frac{f'(t)}{f(t)}(N-1) - \frac{f''(t)}{f'(t)}\right)w'(t) = -(f'(t))^2V(f(t))w^p(t). 
		\end{equation*}
		Next, take $g$ such that
		\[\frac{g'(t)}{g(t)} =\frac{f'(t)}{f(t)}(N-1) - \frac{f''(t)}{f'(t)},\]
		that is,
		\begin{equation*}
			g(t) = \frac{f(t)^{N-1}}{f'(t)}.
		\end{equation*}
		
		Thus
		\begin{equation*}
			\frac{1}{g(t)}\left(g(t)w'\right)' = -(f'(t))^2V(f(t))w^p(t). 
		\end{equation*}
		Rearranging the terms, the proof is complete. 
	\end{proof}
	
	\begin{proof}[\textbf{Proof of Theorem \ref{teorema comport assintotico sol radial}}]
		The proof is split into steps to ease its reading.

		\textbf{Step A. Consider first the case $\mathbf{N\geq 2}$.}

		Let 
		\begin{equation}\label{f magicas dim N>2 e N=2} 
			f(t)=
			\begin{cases}
				e^{-t}, \quad \hspace{5.1em}\text{if }N=2\\
				\left(\frac{1}{(N-2)t+1}\right)^{\frac{1}{N-2}},\quad \text{if }N>2.
			\end{cases}
			t\in [0,+\infty).
		\end{equation}
		Then, $f$ satisfies the hypotheses of Lemma \ref{lema da mudanca magica} and 
		\begin{equation*}
			-\frac{f'(t)}{f(t)^{N-1}}=1.
		\end{equation*} 
		Thus, taking 
		\begin{equation}\label{mudanca de variaveis magica}
			w_\alpha(t) = \frac{1}{\alpha^{\frac{1}{2(p-1)}}}u_{\alpha,rad}(f(t))
		\end{equation}
		one has that
		\begin{equation}\label{eq:problema com a mudanca de variaveis magica dependendo de f}
			-\left(\frac{f(t)^{N-1}}{f'(t)}w'_\alpha \right)'= \alpha^{\frac{1}{2}}f(t)^{N-1}f'(t)(4f(t)(1-f(t)))^\alpha w_\alpha^p(t).
		\end{equation}

		\textbf{Step A.1. Case $\mathbf{N>2}$.} 
		
		Substituting $f(t)$ in \eqref{eq:problema com a mudanca de variaveis magica dependendo de f}, $w_\alpha(t)$ satisfies the following equation for $t$ in $(0,\infty)$:
		\begin{equation}\label{eq pos mudanca de variavel boa N>2}
			\begin{cases}
				-w''_\alpha = \alpha^{\frac{1}{2}}\left(\frac{1}{(N-2)t+1}\right)^{\frac{2(N-1)}{N-2}}\left(4\left(\frac{1}{(N-2)t+1}\right)^{\frac{1}{N-2}}\left(1-\left(\frac{1}{(N-2)t+1}\right)^{\frac{1}{N-2}}\right)\right)^\alpha w_\alpha^p,\\
				w_\alpha(0)= 0 =\lim_{t\rightarrow \infty} w'_\alpha(t).
			\end{cases}
		\end{equation}
		From Propositions \ref{prop limitacao por cima norma do sup sol rad} and \ref{prop limitacao por baixo norma do sup sol rad}  there are constants such that $\mathcal{C}_3 \leq \|w_\alpha\|_{\infty}\leq \mathcal{C}_4 $ for $\alpha> \alpha_0$, where $\alpha_0$ is sufficiently large. Notice that using equation \eqref{eq pos mudanca de variavel boa N>2} one can write $w'_\alpha(t)$ as follows:
		\begin{align}\label{limwlin}
			0\leq w'_\alpha(t) =& \scalebox{1}{$\alpha^{\frac{1}{2}}\int_{t}^{\infty}\left(\frac{1}{(N-2)s+1}\right)^{\frac{2(N-1)}{N-2}}\left(4\left(\frac{1}{(N-2)s+1}\right)^{\frac{1}{N-2}}\left(1-\left(\frac{1}{(N-2)s+1}\right)^{\frac{1}{N-2}}\right)\right)^\alpha w_\alpha^p(s)ds$}\nonumber\\
			&\leq \scalebox{1}{$\alpha^{\frac{1}{2}}\mathcal{C}_4^p\int_{0}^{\infty}\left(\frac{1}{(N-2)s+1}\right)^{\frac{2(N-1)}{N-2}}\left(4\left(\frac{1}{(N-2)s+1}\right)^{\frac{1}{N-2}}\left(1-\left(\frac{1}{(N-2)s+1}\right)^{\frac{1}{N-2}}\right)\right)^\alpha ds$} \nonumber\\
			& =\scalebox{1}{$\alpha^{\frac{1}{2}}\mathcal{C}_4^p \int_{0}^{1}r^{N-1}(4r(1-r))^\alpha dr$}= \frac{\mathcal{C}_4^p \sqrt{\pi}}{2^N} + O(\alpha^{-1}),
		\end{align}
		where Lemma \ref{lemma estimativa dos quocientes com as funcos gama} was used in the last equality. Thus $|w'_\alpha(t)|$ is uniformly bounded and $E_\alpha = \{w_\alpha; \alpha> \alpha_0\}$ is equicontinuous.   Let $(\alpha_n)$ be any sequence such that $\alpha_n\rightarrow \infty$ as $n\rightarrow \infty$; by Arzelà-Ascoli Theorem (see \cite[Theorem~4.44]{Folland1999}), taking a subsequence if necessary, there is a continuous function $w$ such that $w_{\alpha_n}\rightarrow w$ uniformly on any compact subset of $[0,\infty)$. Now one needs to characterize $w$.
		
		Notice that the weight 
		\[\alpha^{\frac{1}{2}}\left(\frac{1}{(N-2)t+1}\right)^{\frac{2(N-1)}{N-2}}\left(4\left(\frac{1}{(N-2)t+1}\right)^{\frac{1}{N-2}}\left(1-\left(\frac{1}{(N-2)t+1}\right)^{\frac{1}{N-2}}\right)\right)^\alpha\]
		appearing in the right hand side of equation \eqref{eq pos mudanca de variavel boa N>2}
		goes to zero uniformly in any interval $[0,a]$ or $[b,\infty)$ with $a<\frac{2^{N-2}-1}{N-2}<b$ and it goes to $\infty$ at $\frac{2^{N-2}-1}{N-2}$, as $\alpha\rightarrow \infty$. Therefore, in order to isolate the singularity and apply Arzelà-Ascoli Theorem again, define the following set.
		
		Given $\rho \in (0,\frac{2^{N-2}-1}{N-2})$, let
		\[ M_\rho \coloneq \left[0,\frac{2^{N-2}-1}{N-2}-\rho\right]\cup\left[\frac{2^{N-2}-1}{N-2}+\rho,\infty\right). \]
		
		By \eqref{eq pos mudanca de variavel boa N>2}, $|w_\alpha''(t)|$ is uniformly bounded for $t\in M_\rho$ and $E'_\alpha|_{M_\rho} =\{w'_\alpha|_{M_\rho}; \alpha> \alpha_0\}$ is equicontinuous.
		
		Similarly, taking a subsequence if necessary, there is a continuous function $g$ (defined on $ M_\rho$)  such that $w'_{\alpha_n} \rightarrow g$ uniformly on any compact subset of $ M_\rho$. Thus, $w'(t)=g(t)$ in $M_\rho$.

		Moreover, from the discussion above, from \eqref{eq pos mudanca de variavel boa N>2}, and fundamental theorem of calculus, it follows that $w''(t)= 0$ on $M_{\rho}$.
		
		Therefore, by the definition of $M_\rho$, $w(t)$ is a straight line in $[0,\frac{2^{N-2}-1}{N-2}-\rho]$ and a straight line in $[\frac{2^{N-2}-1}{N-2}+\rho,\infty)$ for any $\rho \in (0, \frac{2^{N-2}-1}{N-2})$. 
		
		Observe that, for all $t \in[0,\infty)$, $0\leq w(t) \leq \mathcal{C}_4$. Thus in the unbounded interval $[\frac{2^{N-2}-1}{N-2}+\rho,\infty)$ one has that $w(t)$ must be a constant $\tilde{C}$ and in particular  $w'(t)=g(t)\equiv 0$ on $[\frac{2^{N-2}-1}{N-2}+\rho,\infty)$ for any $\rho \in (0, \frac{2^{N-2}-1}{N-2})$. 
		
		For the bounded interval  $[0,\frac{2^{N-2}-1}{N-2}-\rho]$ observe that $w(0)=0$, then $w(t)$ has to be a straight line $\eta t$ where 
		\begin{equation}\label{eq definicao de eta}
			\eta = w'(0)= \lim_{n\rightarrow \infty}w'_{\alpha_n}(0)=-\lim_{n\rightarrow \infty}u'_{\alpha_n,rad}(1)\alpha_n^{-\frac{1}{2(p-1)}}. 
		\end{equation}
		
		By the convergence of $w_{\alpha_n}$ to $w$ as $n\to \infty$ and the arbitrariness of $\rho$, one concludes that
		\begin{equation*}
			w(t) = \begin{cases}
				\eta t, t\in [0,\frac{2^{N-2}-1}{N-2})\\
				\tilde{C}, t\in (\frac{2^{N-2}-1}{N-2},\infty).
			\end{cases}
		\end{equation*}	
		Therefore, since $w$ is continuous on $[0,\infty)$, 
		\begin{equation}\label{igualdade que nos diz que o problema limite cola bem}
			\eta\,\frac{2^{N-2}-1}{N-2} =  \tilde{C}.
		\end{equation}
		
		Next, the explicit values for $\tilde{C}$ and $\eta$ are computed, proving in particular that they do not depend on the subsequences of $(\alpha_n)$. First, we will analyze the behavior of $w'_\alpha(t)$ for $t \in(\frac{2^{N-2}-1}{N-2},\infty)$. Using equation \eqref{eq pos mudanca de variavel boa N>2} and the bound $\|w_{\alpha}\|_{\infty} \leq \mathcal{C}_4$, one has that
		\begin{align}\label{funcao L1 que limita a derivada no intervalo nao limitado}
			0\leq w'_\alpha(t)&\leq \scalebox{1}{$\alpha^{\frac{1}{2}}\mathcal{C}_4^p\int_{t}^{\infty}\left(\frac{1}{(N-2)s+1}\right)^{\frac{2(N-1)}{N-2}}\left(4\left(\frac{1}{(N-2)s+1}\right)^{\frac{1}{N-2}}\left(1-\left(\frac{1}{(N-2)s+1}\right)^{\frac{1}{N-2}}\right)\right)^\alpha ds$}\nonumber\\
			& = \alpha^{\frac{1}{2}}\mathcal{C}_4^p \int_{0}^{\left(\frac{1}{(N-2)t+1}\right)^{\frac{1}{N-2}}} r^{N-1}(4r(1-r))^\alpha dr\nonumber\\
			&\leq  \scalebox{1}{$\alpha^{\frac{1}{2}}\mathcal{C}_4^p\left(\frac{1}{(N-2)t+1}\right)^{\frac{N-1}{N-2}}\left(4\left(\frac{1}{(N-2)t+1}\right)^{\frac{1}{N-2}}\left(1-\left(\frac{1}{(N-2)t+1}\right)^{\frac{1}{N-2}}\right)\right)^\alpha.$}
		\end{align}
		The last inequality follows from the fact that  $r^{N-1}(r(1-r))^\alpha$ is increasing on $[0,\frac{\alpha+N-1}{2\alpha+N-1})$ and $((N-2)t+1)^{-\frac{1}{N-2}}<\frac{\alpha+N-1}{2\alpha+N-1}$ for $t \in(\frac{2^{N-2}-1}{N-2},\infty)$.
		
		Using \eqref{mudanca de variaveis magica} and \eqref{f magicas dim N>2 e N=2}, one obtains 
		\[\int_{0}^{\infty}|w'_{\alpha_n}|^2(t)dt = \frac{1}{|\mathbb{S}^{N-1}|\alpha_n^{\frac 1 {p-1}}}\int_B |Du_{\alpha_n, rad}|^2 =\frac{(S_{\alpha_n,rad})^{\frac{p+1}{p-1}}}{|\mathbb{S}^{N-1}|\alpha_n^{\frac{1}{p-1}}}. \]
		By Theorem \ref{teorema estimativa sharp expoente de S alpha,rad em relacao a alpha} there exists a constant $C$ such that, taking a subsequence if necessary,  $S_{\alpha_n,rad}\, \alpha_n^{-\frac{1}{p+1}} \to C $ as $n$ increases. Thus
		\[\int_{0}^{\infty}|w'_{\alpha_n}|^2(t)dt \approx \frac{C^{\frac{p+1}{p-1}}}{|\mathbb{S}^{N-1}|}. \]
		
		Write
		\[\int_{0}^{\infty}|w'_{\alpha_n}|^2(t)dt= \int_{0}^{\frac{2^{N-2}-1}{N-2}}|w'_{\alpha_n}|^2(t)dt + \int_{\frac{2^{N-2}-1}{N-2}}^{z}|w'_{\alpha_n}|^2(t)dt + \int_{z}^{\infty}|w'_{\alpha_n}|^2(t)dt \]
		with  any $z > \frac{2^{N-2}-1}{N-2} $. Then observe that, using \eqref{funcao L1 que limita a derivada no intervalo nao limitado}, and making the change of variable $s = f(t)$, the last integral goes to zero as $n \rightarrow \infty$ since $\alpha_n^\frac{1}{2} V_{\alpha_n}(s) \to 0 $, as $\alpha_n \to \infty$, uniformly in each interval of the form $[0, -\epsilon +1/2]$. Moreover, from the convergence of $w'_{\alpha_n} \to w'$, from the uniform bound for $w'_{\alpha_n}$ given by \eqref{limwlin}, the fact that $w'\equiv 0$ on $((2^{N-2}-1)/(N-2), z)$ and the dominated convergence theorem, it follows that the second integral also converges to zero. Therefore, as $n\rightarrow \infty$, using again the dominated convergence theorem on the remaining integral one obtains that
		\[\int_{0}^{\frac{2^{N-2}-1}{N-2}}\eta^2dt = \frac{C^{\frac{p+1}{p-1}}}{|\mathbb{S}^{N-1}|} \]
		implying that 
		\begin{equation}\label{eq relacao entre eta e a cte de S_alpha,rad}
			\eta = \sqrt{\frac{C^{\frac{p+1}{p-1}}(N-2)}{\left(2^{N-2}-1\right)|\mathbb{S}^{N-1}|}}.
		\end{equation}
		Now, using \eqref{mudanca de var util usada na aprx da derivada em 1} and \eqref{mudanca de variaveis magica}, write 
		\[w'_{\alpha_n}(0)= -\frac{u'_{\alpha_n,rad}(1)}{{\alpha_n}^{\frac{1}{2(p-1)}}}= -\frac{y'_{\alpha_n}(1)}{\sqrt{\pi}^{\frac{1}{p-1}}},\]
		and by Corollary \ref{corolario util na aprox da derivada em 1} one has that 
		\begin{equation*}
			w'_{\alpha_n}(0) \approx \frac{1}{\sqrt{\pi}^{\frac{1}{p-1}}} \left(\frac{\alpha_n^{\frac{1}{2}}}{2^N\sqrt{\pi}} \right) y_{\alpha_n}^p\left(\frac{1}{2}\right)\int_{-1}^{1}(1-t^2)^{\alpha_n}(1-t)dt.
		\end{equation*}
		Since Lemma \ref{lemma estimativa dos quocientes com as funcos gama} implies that $\int_{-1}^{1}(1-t^2)^{\alpha_n}(1-t)dt = \sqrt{\pi}\alpha_n^{-\frac{1}{2}}+ O(\alpha_n^{-\frac{3}{2}})$, and 
		\[ y_{\alpha_n}^p\left(\frac{1}{2}\right)= \sqrt{\pi}^{\frac{p}{p-1}}w^p_{\alpha_n}\left(\frac{2^{N-2}-1}{N-2}\right), \]
		one concludes
		\begin{equation*}
			w'_{\alpha_n}(0) \approx \frac{\sqrt{\pi}}{2^N}w^p_{\alpha_n}\left(\frac{2^{N-2}-1}{N-2}\right)(1 + O(\alpha_n^{-1})).
		\end{equation*}
		Due to \eqref{eq definicao de eta} and the continuity of $w$ in $(0, \infty)$, letting $n\rightarrow \infty$: 
		\begin{equation}\label{realcao eta e c til 2}
			\eta = \frac{\sqrt{\pi}}{2^N} \tilde{C}^p.
		\end{equation}
		Thus using \eqref{igualdade que nos diz que o problema limite cola bem} and \eqref{realcao eta e c til 2} we have
		\begin{equation}
			\tilde{C} = \left(\frac{2^N(N-2)}{\sqrt{\pi}(2^{N-2}-1)}\right)^{\frac{1}{p-1}}
		\end{equation}
		and then
		\begin{equation}
			\eta = \left(\frac{2^N}{\sqrt{\pi}}\left(\frac{N-2}{2^{N-2}-1}\right)^p \right)^{\frac{1}{p-1}}.
		\end{equation}
		Finally, using \eqref{eq relacao entre eta e a cte de S_alpha,rad} one obtains
		\begin{equation}\label{cte S_alpha,rad caso N>2}
			C = \left(\frac{N-2}{2^{N-2}-1}\right)\left(\frac{4^N}{\pi}|\mathbb{S}^{N-1}|^{p-1}\right)^{\frac{1}{p+1}}.
		\end{equation}

		\textbf{Step A.2. Case $\mathbf{N=2}$.} 
		
		The arguments follow as in case $N>2$, defining $w_\alpha$ as in \eqref{mudanca de variaveis magica} with $f$ defined in  
		\eqref{f magicas dim N>2 e N=2}. The constants for $N=2$ are
		\begin{equation*}
			\tilde{C} = \left(\frac{4}{\sqrt{\pi}\ln(2)}\right)^{\frac{1}{p-1}}
		\end{equation*}
		\begin{equation*}
			\eta = \left(\frac{4}{\sqrt{\pi}\ln(2)^p}\right)^{\frac{1}{p-1}}
		\end{equation*}
		\begin{equation}\label{cte S_alpha,rad caso N=2}
			C = \frac{1}{\ln(2)}\left(\frac{16}{\pi}(2\pi)^{p-1}\right)^{\frac{1}{p+1}}.
		\end{equation}
		
		\textbf{Step B. The case $\mathbf{N=1}$.}

		The arguments for this case are slightly simpler. Here, define
		\begin{equation}\label{mudanca magica caso N=1}
			w_\alpha(t)= \frac{1}{\alpha^{\frac{1}{2(p-1)}}}u_{\alpha,rad}(1-t),
		\end{equation}
		for $t\in[0,1]$.																																											 Then 
		\begin{equation}\label{caso N=1 comportamento assint sol rad}
			\begin{cases*}
				-w''_\alpha(t) = \alpha^{\frac{1}{2}}(4t(1-t))^\alpha w^p_\alpha(t)\\
				w_\alpha(0)=0 = w'_\alpha(1)
			\end{cases*}	
		\end{equation}
		From Proposition \ref{prop limitacao por cima norma do sup sol rad} and Proposition \ref{prop limitacao por baixo norma do sup sol rad} one has that, for $\alpha$ sufficiently large, $\mathcal{C}_3\leq \|w_\alpha\|_\infty\leq \mathcal{C}_4$. Also notice that
		\begin{align*}
			w'_\alpha(t) &= \alpha^{\frac{1}{2}}\int_{t}^{1}(4s(1-s))^\alpha w_\alpha^p(s)ds\leq \alpha^{\frac{1}{2}} \mathcal{C}_4^p\int_{0}^{1}(4s(1-s))^\alpha ds\\
			&= \alpha^{\frac{1}{2}} \mathcal{C}_4^p \frac{4^\alpha \Gamma(\alpha+1)^2}{\Gamma(2\alpha +2)}= \frac{\sqrt{\pi}}{2}\mathcal{C}_4^p + O(\alpha^{-1}).
		\end{align*}
		Let $(\alpha_n)$ be any sequence such that $\alpha_n \rightarrow \infty$ as $n\rightarrow \infty$. By Arzelà-Ascoli theorem, taking a subsequence if necessary, there exists a continuous function $w$ on $[0,1]$ such that $w_{\alpha_n}\rightarrow w$ uniformly on $[0,1]$. Arguing as in the case $N > 2$, there are continuous functions such that, taking subsequences if necessary, $w'_{\alpha_n}(t) \rightarrow g$ and $w''_{\alpha_n} \rightarrow \kappa$ uniformly on $M_\rho \coloneq [0,\frac{1}{2}-\rho]\cup [\frac{1}{2}+\rho,1]$. Thus, $w'(t) =g(t)$ and $w''(t)=\kappa(t)$ for $t\in M_\rho$. Using \eqref{caso N=1 comportamento assint sol rad} one concludes that $\kappa \equiv 0$. Since $w(0)=0$ and $w(1)=\|w\|_\infty =: \tilde C$, one has that	
		\begin{equation*}
			\begin{cases*}
				w(t)= \eta t, \quad\text{for } t\in[0,\frac{1}{2}]\\
				w(t)= m(1-t)+ \tilde{C},\quad \text{for } t\in [\frac{1}{2},1],
			\end{cases*}
		\end{equation*}
		where 
		\begin{equation}\label{newderivada1}
			\eta = \lim_{\alpha\rightarrow \infty} \frac{-u'_{\alpha,rad}(1)}{\alpha^{\frac{1}{2(p-1)}}}
		\end{equation}
		and $m= \eta - 2\tilde{C}$, because $w$ is continuous. Since $w'_{\alpha}(1)= -\frac{1}{\alpha^{\frac{1}{2(p-1)}}}u'_{\alpha,rad}(0)=0$, it follows that $0=w'(1)=-(\eta-2\tilde{C})$. Thus 
		\begin{equation}\label{eq relacao entre eta e tilde C caso N=1}
			\eta = 2 \tilde{C}
		\end{equation}
		and
		\begin{equation*}
			\begin{cases*}
				w(t)= \eta t, \quad\text{for } t\in[0,\frac{1}{2}]\\
				w(t)= \tilde{C},\quad \text{for } t\in [\frac{1}{2},1].
			\end{cases*}
		\end{equation*}
		Using \eqref{mudanca magica caso N=1}, 
		\begin{equation*}
			\int_{0}^{1}|w'_{\alpha_n}|^2(t) dt =\frac{\left(S_{\alpha_n,rad}\right)^{\frac{p+1}{p-1}}}{2\alpha_n^{\frac{1}{p-1}}}.
		\end{equation*}
		By Theorem \ref{teorema estimativa sharp expoente de S alpha,rad em relacao a alpha} there exists a constant $C$ such that, taking a subsequence if necessary,  $S_{\alpha_n,rad}\, \alpha_n^{-\frac{1}{p+1}} \to C $ as $n$ increases. Thus
		
		\begin{equation*}
			\int_{0}^{1}|w'_{\alpha_n}|^2(t) dt \approx \frac{C^{\frac{p+1}{p-1}}}{2}
		\end{equation*}
		and, as $n\rightarrow \infty$, one has that
		\begin{equation*}
			\int_{0}^{\frac{1}{2}}\eta^2dt = \frac{C^{\frac{p+1}{p-1}}}{2},
		\end{equation*}
		therefore
		\begin{equation}\label{eq relacao entre eta e a cte de S_alpha,rad no caso N=1}
			\eta^2 = C^{\frac{p+1}{p-1}}.
		\end{equation}
		
		Using the approximation Lemma \ref{lemma util que usa a aproximacao de weierstrass} and Corollary \ref{corolario util na aprox da derivada em 1}, it follows that
		\[w'_{\alpha_n}(0) \approx \frac{\sqrt{\pi}}{2}w_{\alpha_n}^p\left(\frac{1}{2}\right)\left(1+ O(\alpha_n^{-1})\right).\]
		Thus, making $n\rightarrow \infty$, one obtains
		\begin{equation*}
			\eta = \frac{\sqrt{\pi}}{2}\frac{\eta^p}{2^p}.
		\end{equation*}
		Therefore
		\begin{equation*}
			\eta = \left(\frac{2^{p+1}}{\sqrt{\pi}}\right)^{\frac{1}{p-1}}.	
		\end{equation*}
		Using \eqref{eq relacao entre eta e tilde C caso N=1} it follows
		\begin{equation*}
			\tilde{C} = \left(\frac{4}{\sqrt{\pi}}\right)^{\frac{1}{p-1}}
		\end{equation*}
		and from \eqref{eq relacao entre eta e a cte de S_alpha,rad no caso N=1} one obtains
		\begin{equation}\label{cte S_alpha,rad caso N=1}
			C = \frac{4}{\pi^{\frac{1}{p+1}}}.
		\end{equation}

		\textbf{Step C. Uniform convergence on $\mathbf{[0,1]}$.}	
		
		Notice that, reversing the change of variables, one can write
		\begin{equation}
			\label{eq : desfazendo a mudanca magica}
			\frac{u_{\alpha,rad}(t)}{\alpha^{\frac{1}{2(p-1)}}} = \begin{cases}
				w_\alpha\left((N-2)^{-1}(t^{2-N}-1)\right)\quad \text{if }N>2,\\
				w_\alpha(-\ln(t)) \quad \text{if }N=2,\\
				w_{\alpha}(1-t)\quad \text{if } N=1.
			\end{cases} 
		\end{equation}
		Thus, by the convergences established for $w_\alpha$ in Steps A and B, \eqref{eq : desfazendo a mudanca magica} gives the uniform convergence on each compact set of $(0,1]$ for $N\geq 2$ and on $[0,1]$ for $N=1$. 
		
		To obtain the uniform convergence on $[0,1]$ for $N\geq 2$, one needs to investigate the behavior of $u_{\alpha,rad}(0)\alpha^{-\frac{1}{2(p-1)}}$ as $\alpha \rightarrow \infty$. It holds that 
		\begin{equation}
			\label{eq convergencia de u_alpha no 0}
			\frac{u_{\alpha,rad}(0)}{\alpha^{\frac{1}{2(p-1)}}} \to \tilde{C} \quad {as} \quad \alpha \to \infty.
		\end{equation}
		
		Indeed, write the integral representation for $u_{\alpha,rad}(0)$ as in \eqref{eq: green function integral representation}, namely
		\begin{multline}\label{eq u_alpha,rad(0) usando green}
			u_{\alpha,rad}(0)\alpha^{-\frac{1}{2(p-1)}}= \alpha^{\frac{1}{2}}\int_0^1 s^{N-1}G_N(0,s)(4s(1-s))^\alpha\left(u_{\alpha,rad}(s)\alpha^{-\frac{1}{2(p-1)}} \right)^p ds\\
			= \frac{\sqrt{\pi}}{2}\int_{-1}^1 \left(\frac{r+1}{2}\right)^{N-1}G_N\left(0,\frac{r+1}{2}\right)\left(u_{\alpha,rad}\left(\frac{r+1}{2}\right)\alpha^{-\frac{1}{2(p-1)}} \right)^p \varphi_\alpha(r)dr 
		\end{multline}    
		with $\varphi_\alpha$ defined as in \eqref{phialpha}. For $s\in [0,2]$ define 
		\begin{equation*}
			f_\alpha(s) = \left(\frac{s}{2}\right)^{N-1}G_N\left(0,\frac{s}{2}\right)\left(u_{\alpha,rad}\left(\frac{s}{2}\right)\alpha^{-\frac{1}{2(p-1)}} \right)^p.
		\end{equation*}
		
		Then, by \eqref{eq u_alpha,rad(0) usando green}, it follows that
		\begin{equation}\label{eq u alpha rad (0) vou passar o limite}
			u_{\alpha,rad}(0)\alpha^{-\frac{1}{2(p-1)}} = \frac{\sqrt{\pi}}{2}\int_{-1}^{1}f_\alpha(r+1)\varphi_\alpha(r) dr.
		\end{equation}
		
		Using \eqref{eq : desfazendo a mudanca magica} one can rewrite $f_\alpha$ as
		\begin{equation}
			f_\alpha(s) = \left(\frac{s}{2}\right)^{N-1}G_N\left(0,\frac{s}{2}\right)w_{\alpha}^p\left(f^{-1}\left(\frac{s}{2}\right)\right),
		\end{equation}
		where $f$ is the change of variables \eqref{f magicas dim N>2 e N=2}, depending on the choice of $N$. Define
		\begin{equation}\label{eq def de f w}
			f_w(s)=\left(\frac{s}{2}\right)^{N-1}G_N\left(0,\frac{s}{2}\right)w^p\left(f^{-1}\left(\frac{s}{2}\right)\right) \quad s\in[0,2].
		\end{equation}
		
		Thus, by the uniform convergence of $w_\alpha$ to $w$ on compact sets $K\subset [0,\infty)$, one has that $f_\alpha(s) \rightarrow f_w(s)$ uniformly on compact sets of the form $2f(K)\subset (0,2]$. Proposition \ref{prop limitacao por cima norma do sup sol rad} implies that $\|w_\alpha\|_{\infty}\leq \mathcal{C}_4$ therefore there exists a constant $m>0$ such that $\|f_\alpha\|_\infty \leq m$. 
		
		Take $\sigma= \frac{1}{2}$, then, given $\epsilon>0$, there exist $\alpha_1>0$ such that $\alpha>\alpha_1$ implies $|f_\alpha(x)-f_w(x)|< \frac{\epsilon}{24}$ for all $x\in [1-\sigma,1+\sigma]$. By the continuity of $f_w$, there exists $\tilde{\sigma}>0$ such that, if $x,y\in [1-\sigma,1+\sigma]$ and $|x-y|\leq\tilde{\sigma}$ then $|f_w(x)-f_w(y)|< \frac{\epsilon}{24}$. Let $\delta = \min\{\sigma,\tilde{\sigma}\}$.
		
		As in Lemma \ref{lemma aprox de weierstrass} one has that $\varphi_\alpha(x) \rightarrow 0$ uniformly for $|x|\in [\delta,1]$ and
		\[g(\alpha)=\int_{-1}^{1}\varphi_\alpha(s)ds-1 = O(\alpha^{-1}) \quad\text{as }\alpha\to \infty.\]
		Let $\alpha_2>0$ be such that $\alpha>\alpha_2$ implies $|\varphi_\alpha(t)|<\frac{\epsilon}{8m}$ for $|t|\geq \delta$ and $|g(\alpha)|<\min\{1,\frac{\epsilon}{8m}\}$.
		
		Define
		\begin{equation*}
			P_\alpha(1)= \int_{-1}^{1}f_\alpha(r+1)\varphi_\alpha(r) dr.
		\end{equation*}
		Then, for $\alpha > \max\{\alpha_1,\alpha_2\}$, proceeding as in the proof of Lemma \ref{lemma aprox de weierstrass}, one has that
		
		\[|f_\alpha(1)-P_\alpha(1)|< \epsilon.\]
		
		Therefore
		\begin{equation}\label{eq util para passar o lmite em u alpha rad(0)}
			\lim_{\alpha\rightarrow \infty} P_\alpha(1) = \lim_{\alpha\rightarrow \infty} f_\alpha(1) = f_w(1)= \frac{1}{2^{N-1}}G_N\left(0,\frac{1}{2}\right)w^p\left(f^{-1}\left(\frac{1}{2}\right)\right).
		\end{equation}
		From \eqref{eq u alpha rad (0) vou passar o limite}, \eqref{eq util para passar o lmite em u alpha rad(0)} and \eqref{igualdade que nos diz que o problema limite cola bem} one concludes
		\begin{multline*}
			\lim_{\alpha\rightarrow \infty}	u_{\alpha,rad}(0)\alpha^{-\frac{1}{2(p-1)}} = \frac{\sqrt{\pi}}{2^N}G_N\left(0,\frac{1}{2}\right)w^p\left(f^{-1}\left(\frac{1}{2}\right)\right) = \frac{\sqrt{\pi}}{2^N}G_N\left(0,\frac{1}{2}\right)\tilde{C}^p\\
			= \eta G_N\left(0,\frac{1}{2}\right) = \tilde{C}.
		\end{multline*}
		Since $u_{\alpha,rad}(t)\alpha^{-\frac{1}{2(p-1)}}$ is decreasing on $[0,1]$, we obtain the uniform convergence of $u_{\alpha,rad}(t)\alpha^{-\frac{1}{2(p-1)}}$ on the compact set $[0,1]$ for $N\geq 2$.
		
		Finally, observe that \eqref{teoderivada1} follows from \eqref{newderivada1} in the case $N=1$, \eqref{eq definicao de eta} if $N >2$, and by adapting arguments used to prove \eqref{eq definicao de eta} to the case $N=2$.
	\end{proof}

	\begin{proof}[\textbf{Proof of Theorem \ref{teo caracterizacao de S alpha rad}}]
		This is a consequence of Theorem \ref{teorema estimativa sharp expoente de S alpha,rad em relacao a alpha} combined with: equation \eqref{eq relacao entre eta e a cte de S_alpha,rad} and the explicit values obtained in \eqref{cte S_alpha,rad caso N>2} and \eqref{cte S_alpha,rad caso N=2}, for $N\geq 2$; equation \eqref{eq relacao entre eta e a cte de S_alpha,rad no caso N=1} and the explicit value obtained in \eqref{cte S_alpha,rad caso N=1} for $N=1$.
	\end{proof}
	
	\begin{proof}[\textbf{Proof of Corollary \ref{corl: escrevendo a convg em termos da funcao de green}}]
		\hypertarget{proofCorolGreen}{}

		The convergence
		\begin{equation*}
			\frac{u_{\alpha,rad}}{\alpha^{\frac{1}{2(p-1)}}}(t) \rightarrow \eta G_N\left(t,\frac{1}{2}\right) 
		\end{equation*}
		is an immediate consequence of calculating $\eta G_N\left(t,\frac{1}{2}\right)$ and applying Theorem \ref{teorema comport assintotico sol radial}. 
	\end{proof}
	\begin{remk}
		Notice that $G_N$ here differs from $G$ defined in  \cite{Singh2014} by a minus sign because the equation in the present paper involves $-\Delta$ instead of $\Delta$.
	\end{remk}

	\section{Open problems}\label{sec: open problems}
	We leave the following as open problems. 
	
	\textbf{Open Problem 1. } Investigating the radial symmetry of ground state solutions in the case $p(N-2)\leq N$. The restriction $p >\frac{N}{N-2}$ in Theorem \ref{teorema quebra de simetria} is needed in order to guarantee that $\frac{1}{p+1}>1+ \frac{N}{p+1}-\frac{N}{2}$. 
	
	\textbf{Open Problem 2.  } In view of Theorem \ref{teorema solucoes decrescem fora da bola de raio meio}, since the maximum of $V_{\alpha}$ is attained at $1/2$, we conjecture that ground state solutions of problem \eqref{problema-principal} concentrate at points on the sphere of radius $1/2$ as $\alpha \to \infty$.

	\section{Numerical analysis of radial solutions}\label{sec: analise numerica}
	
	In this section the Runge-Kutta-Oliver method developed in \cite{Tsitouras2019} is applied to equation \eqref{problema-principal} in the radial setting, by taking the partition $\{x_i = \frac{i}{k} \text{ for } i=0,\ldots, k \}$ of $[0,1]$ with $k= 1000$, $N=3$, $p=3$ and various values for $\alpha$. Figure \ref{fig:solucao diferentes alphas} shows the behavior of the normalized solution of problem \eqref{problema-principal} as $\alpha$ increases. The behavior, as expected, follows the asymptotics of Theorem \ref{teorema comport assintotico sol radial}.  Figure \ref{fig:derivada da solucao diferentes alphas} gives the behavior of the \emph{derivative} of the normalized solution with different values of $\alpha$.

	\begin{figure}[H]
		\centering
		\includegraphics[width=0.8\textwidth,trim={0 0.13cm 0 0.1cm},clip]{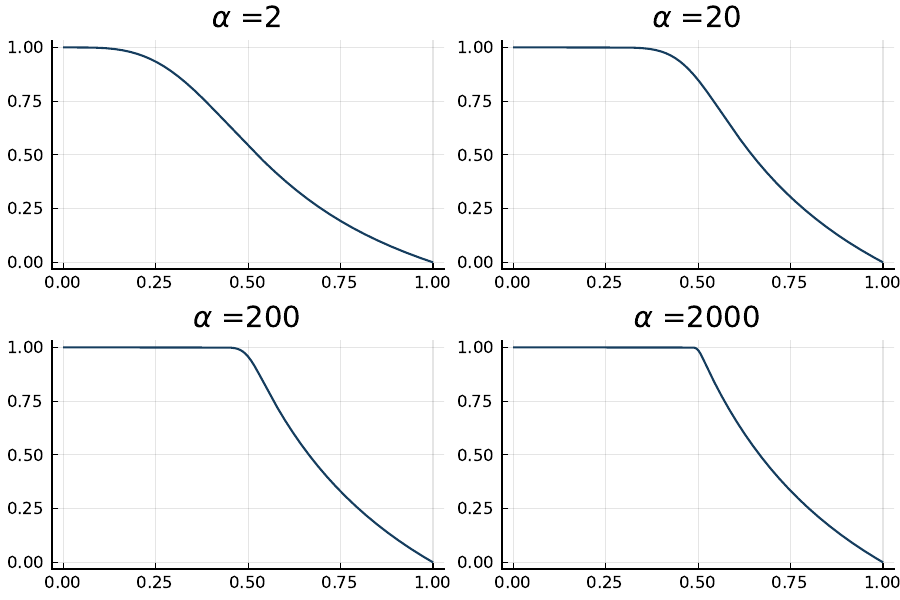}
		\caption{Asymptotic behavior of the normalized solution with $N=3,p=3$ and different values of $\alpha$.}
		\label{fig:solucao diferentes alphas}
	\end{figure}

\begin{figure}[H]
		\centering
		\includegraphics[width=0.8\textwidth,trim={0 0.13cm 0 0.1cm},clip]{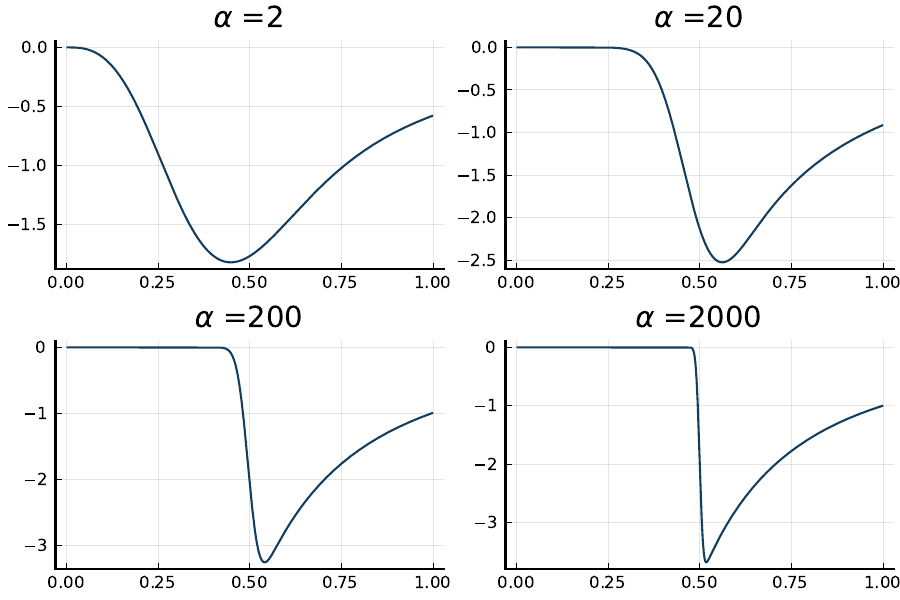}
		\caption{Asymptotic behavior of the derivative of the normalized solution with $N=3,p=3$ and different values of $\alpha$.}
		\label{fig:derivada da solucao diferentes alphas}
	\end{figure}

   Table \ref{tab: tabela S_alpha_rad X alpha E u0 X alpha} gives some approximate values for $S_{\alpha,rad}\alpha^{-\frac{1}{p+1}}$ and $u_{\alpha,rad}(0)\alpha^{-\frac{1}{2(p-1)}}$ as $\alpha$ increases, using the numerical solution as in the previous paragraph. For $\alpha \to \infty$, the values $C$ and $\tilde{C}$ are obtained in Theorem \ref{teo caracterizacao de S alpha rad} and Theorem \ref{teorema comport assintotico sol radial}. 

      Figures \ref{fig:S_alpha,rad x alpha} and \ref{fig:u0 x alpha} show these values as a function of $\alpha$. For the Hénon equation, in \cite[Theorem~4.1]{SMETS2002}, the authors have shown that the best constant $C$ associated with  $S_\alpha^R$ ($S_{\alpha,rad}$ in this paper) is approached monotonically as $\alpha$ increases. Figure \ref{fig:S_alpha,rad x alpha} indicates that the same occurs with the weight in this paper, that is, $C = \inf_{\alpha>0} S_{\alpha,rad}\alpha^{-\frac{1}{p+1}}$. Figure \ref{fig:u0 x alpha} indicates that $\tilde{C}$ behaves similarly, that is, $\tilde{C} = \inf_{\alpha>0}u_{\alpha,rad}(0)\alpha^{-\frac{1}{2(p-1)}}$.
    \begin{table}[H]
		\centering
		\begin{tabular}{ccc}
			
			\toprule
			$\alpha$ & $S_{\alpha,rad}\alpha^{-\frac{1}{p+1}}$ & $u_{\alpha,rad}(0)\alpha^{-\frac{1}{2(p-1)}}$\\
			
			\midrule 
			2 & 10.7111&4.7755\\

			20 & 8.7729&2.8105\\
			
			200& 7.9522&2.3164\\
			
			2000 & 7.6687&2.1826\\
			
			\vdots   &		\vdots		&\vdots	\\
			
			$\infty$ & 7.5312&2.1245\\
			\bottomrule
		\end{tabular}
		\caption{ Numerical values of $S_{\alpha,rad}\alpha^{-\frac{1}{p+1}}$ and $u_{\alpha,rad}(0)\alpha^{-\frac{1}{2(p-1)}}$ for different values of $\alpha$, obtained by taking the partition $\{x_i = \frac{i}{k} \text{ for } i=0,\ldots, k \}$ of $[0,1]$ with $k= 1000$, $N=3$, $p=3$, and applying \cite{Tsitouras2019}. } 
		\label{tab: tabela S_alpha_rad X alpha E u0 X alpha}
	\end{table}

\begin{figure}[H]
		\centering
		\includegraphics[width=0.8\textwidth,trim={0 0.15cm 0 0.1cm},clip]{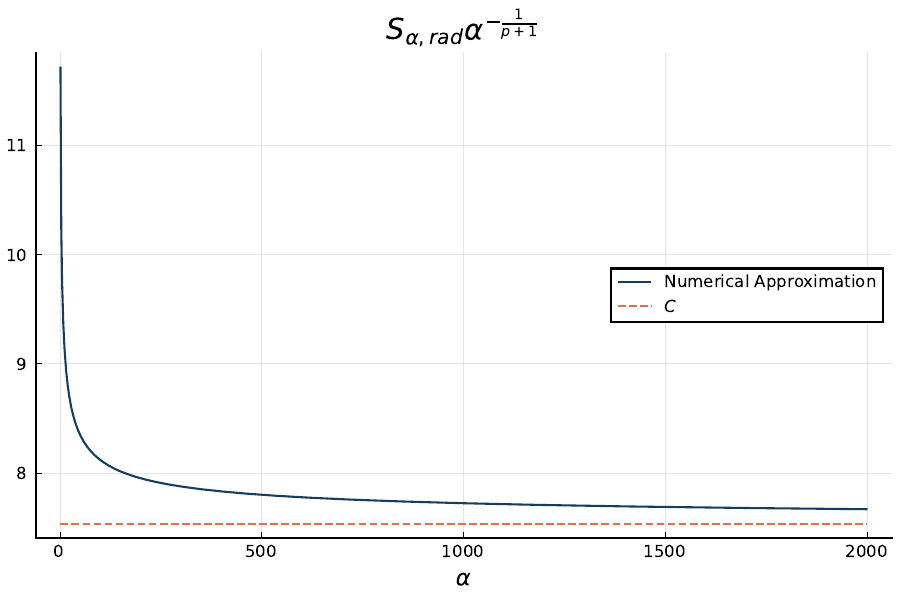}
		\caption{Behavior of $S_{\alpha,rad}\alpha^{-\frac{1}{p+1}}$ with $N=3,p=3$ as $\alpha$ increases.}
		\label{fig:S_alpha,rad x alpha}
	\end{figure}
	\begin{figure}[H]
		\centering
		\includegraphics[width=0.8\textwidth,trim={0 0.15cm 0 0.1cm},clip]{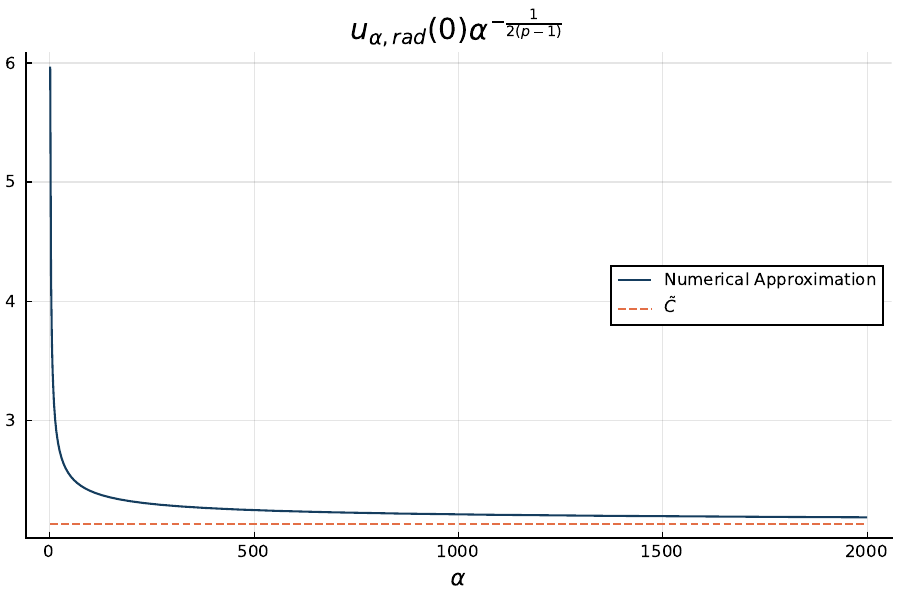}
		\caption{Behavior of $u_{\alpha,rad}(0)\alpha^{-\frac{1}{2(p-1)}}$ with $N=3,p=3$ as $\alpha$ increases.}
		\label{fig:u0 x alpha}
	\end{figure}


	\section*{Acknowledgments}
	
	Pablo dos Santos Corrêa Junior is supported by the Coordenação de Aperfeiçoamento de Pessoal de Nível Superior - Brasil (CAPES) - Finance Code 001.
	
	Ederson Moreira dos Santos is partially supported by CNPq grant 312867/2023-9 and FAPESP grant 2022/16407-1.
	
	Delia Schiera is partially supported by the Portuguese government through FCT Fundação para a Ciência e a Tecnologia and the Recovery and Resilience Plan (PRR) through projects UID/04459/2025 and UID/PRR/04459/2025 (CAMGSD), 2023.17881.ICDT with DOI  identifier 10.54499/2023.17881.ICDT (project SHADE), 2024.14494.PEX, DOI  10.54499/2024.14494.PEX (project ASSO), and through the FCT Mobility grant FCT/Mobility/1304796440/2024-25.
	She is also supported by GNAMPA (Gruppo Nazionale per l’Analisi, Probabilità e le loro Applicazioni) – INdAM (Istituto
	Nazionale di Alta Matematica), through the project ‘Critical and limiting phenomena in nonlinear elliptic
	systems’, CUP E5324001950001, and by FCT 
	under the Scientific Employment Stimulus - Individual Call (CEEC Individual), DOI identifier  10.54499/2020.02540.CEECIND/CP1587/CT0008.

	\bibliographystyle{plain}
	\bibliography{referencias-edp.bib}

@Book{2010_Evans_BOOK,
  author    = {Evans, Lawrence},
  publisher = {American Mathematical Society},
  title     = {Partial differential equations},
  year      = {2010},
  address   = {Providence, R.I},
  isbn      = {9780821849743},
}

@Book{2007_Lima_BOOK,
  author    = {Lima, Elon},
  publisher = {Instituto de Matem\'atica Pura e Aplicada, CNPq},
  title     = {Espa\c{c}os m\'etricos},
  year      = {2007},
  address   = {Rio de Janeiro},
  isbn      = {9788524401589},
  edition   = {5},
  series    = {Projeto Euclides},
  }

@Book{2011_Brezis_BOOK,
  author    = {Br{é}zis, {H}aïm},
  publisher = {Springer},
  title     = {Functional analysis, {S}obolev spaces and partial differential equations},
  year      = {2011},
  address   = {New York London},
  isbn      = {9780387709130},
}

@Article{Henon1974,
  author    = {Michel H{\'{e}}non},
  journal   = {Symposium - International Astronomical Union},
  title     = {Numerical {E}xperiments on the {S}tability of {S}pherical {S}tellar {S}ystems},
  year      = {1974},
  pages     = {259--259},
  volume    = {62},
  doi       = {10.1017/s0074180900070662},
  publisher = {Cambridge University Press ({CUP})},
}

@Article{Ni1982,
  author   = {Ni, Wei-Ming},
  journal  = {Indiana University Mathematics Journal},
  title    = {A nonlinear {Dirichlet} problem on the unit ball and its applications},
  year     = {1982},
  issn     = {0022-2518},
  pages    = {801--807},
  volume   = {31},
  doi      = {10.1512/iumj.1982.31.31056},
  language = {English},
  zbl      = {0515.35033},
  zbmath   = {3814144},
}

@Article{Cao2008,
  author    = {Daomin Cao and Shuangjie Peng and Shusen Yan},
  journal   = {{IMA} Journal of Applied Mathematics},
  title     = {Asymptotic behaviour of ground state solutions for the Henon equation},
  year      = {2008},
  month     = {dec},
  number    = {3},
  pages     = {468--480},
  volume    = {74},
  doi       = {10.1093/imamat/hxn035},
  publisher = {Oxford University Press ({OUP})},
}

@Article{Gidas1979,
  author   = {Gidas, Basilis and Ni, Wei-Ming and Nirenberg, Louis},
  journal  = {Communications in Mathematical Physics},
  title    = {Symmetry and related properties via the maximum principle},
  year     = {1979},
  issn     = {0010-3616},
  pages    = {209--243},
  volume   = {68},
  doi      = {10.1007/BF01221125},
  language = {English},
  zbl      = {0425.35020},
  zbmath   = {3661841},
}

@Article{Bartsch2005,
  author    = {Bartsch, Thomas and Weth, Tobias and Willem, Michel},
  journal   = {Journal d’Analyse Mathématique},
  title     = {Partial symmetry of least energy nodal solutions to some variational problems},
  year      = {2005},
  issn      = {1565-8538},
  month     = dec,
  number    = {1},
  pages     = {1--18},
  volume    = {96},
  doi       = {10.1007/bf02787822},
  publisher = {Springer Science and Business Media LLC},
}

@Article{Tricomi1951,
  author   = {Tricomi, Francesco G. and Erd{\'e}lyi, Arthur},
  journal  = {Pacific Journal of Mathematics},
  title    = {The asymptotic expansion of a ratio of gamma functions},
  year     = {1951},
  issn     = {1945-5844},
  pages    = {133--142},
  volume   = {1},
  doi      = {10.2140/pjm.1951.1.133},
  language = {English},
  zbl      = {0043.29103},
  zbmath   = {3066945},
}

@Article{SMETS2002,
  author    = {Smets, Didier and Willem, Michel and Su, Jibao},
  journal   = {Communications in Contemporary Mathematics},
  title     = {NON-RADIAL GROUND STATES FOR THE {H}éNON EQUATION},
  year      = {2002},
  issn      = {1793-6683},
  month     = aug,
  number    = {03},
  pages     = {467--480},
  volume    = {04},
  doi       = {10.1142/s0219199702000725},
  publisher = {World Scientific Pub Co Pte Lt},
}

@Article{Byeon2006,
  author    = {Byeon, Jaeyoung and Wang, Zhi-Qiang},
  journal   = {Annales de l’Institut Henri Poincaré C, Analyse non linéaire},
  title     = {On the {H}énon equation: asymptotic profile of ground states, {I}},
  year      = {2006},
  issn      = {1873-1430},
  month     = dec,
  number    = {6},
  pages     = {803--828},
  volume    = {23},
  doi       = {10.1016/j.anihpc.2006.04.001},
  publisher = {European Mathematical Society - EMS - Publishing House GmbH},
}

@Article{LeitedaSilva2021,
  author    = {Leite da Silva, Wendel and Moreira dos Santos, Ederson},
  journal   = {Journal of Differential Equations},
  title     = {Asymptotic profile and {M}orse index of the radial solutions of the {H}énon equation},
  year      = {2021},
  issn      = {0022-0396},
  month     = jun,
  pages     = {212--235},
  volume    = {287},
  doi       = {10.1016/j.jde.2021.03.050},
  publisher = {Elsevier BV},
}

@Article{Mercuri2019,
  author    = {Mercuri, Carlo and Moreira dos Santos, Ederson},
  journal   = {Nonlinearity},
  title     = {Quantitative symmetry breaking of groundstates for a class of weighted {E}mden–{F}owler equations},
  year      = {2019},
  issn      = {1361-6544},
  month     = oct,
  number    = {11},
  pages     = {4445--4464},
  volume    = {32},
  doi       = {10.1088/1361-6544/ab2d6f},
  publisher = {IOP Publishing},
}

@Book{Artin2015,
  author    = {Artin, Emil},
  editor    = {Michael Butler},
  publisher = {Dover Publications, Incorporated},
  title     = {The Gamma Function},
  year      = {2015},
  address   = {Newburyport},
  edition   = {1st ed.},
  isbn      = {9780486803005},
  series    = {Dover Books on Mathematics Series},
  pagetotal = {172},
  ppn_gvk   = {1881221911},
}

@Book{Struwe2010,
  author    = {Struwe, Michael},
  publisher = {Springer},
  title     = {Variational methods},
  year      = {2010},
  address   = {Berlin [u.a.]},
  edition   = {4. ed., pbk.},
  isbn      = {9783642093296},
  number    = {Folge 3, 34},
  series    = {Ergebnisse der Mathematik und ihrer Grenzgebiete},
  pagetotal = {302},
  ppn_gvk   = {64089478X},
  subtitle  = {Applications to nonlinear partial differential equations and Hamiltonian systems},
}

@Article{Bonheure2008,
  author   = {Bonheure, Denis and Serra, Enrico and Tarallo, Massimo},
  journal  = {Advances in Differential Equations},
  title    = {Symmetry of extremal functions in {Moser}-{Trudinger} inequalities and a {H{\'e}non} type problem in dimension two},
  year     = {2008},
  issn     = {1079-9389},
  number   = {1-2},
  pages    = {105--138},
  volume   = {13},
  language = {English},
  zbl      = {1173.35044},
  zbmath   = {5372150},
}

@Article{Gidas1981,
  author    = {Gidas, Basilis and Spruck, Joel},
  journal   = {Communications on Pure and Applied Mathematics},
  title     = {Global and local behavior of positive solutions of nonlinear elliptic equations},
  year      = {1981},
  issn      = {1097-0312},
  month     = jul,
  number    = {4},
  pages     = {525--598},
  volume    = {34},
  doi       = {10.1002/cpa.3160340406},
  publisher = {Wiley},
}

@Article{Wong1975,
  author    = {Wong, James S. W.},
  journal   = {SIAM Review},
  title     = {On the Generalized {E}mden–{F}owler Equation},
  year      = {1975},
  issn      = {1095-7200},
  month     = apr,
  number    = {2},
  pages     = {339--360},
  volume    = {17},
  doi       = {10.1137/1017036},
  publisher = {Society for Industrial & Applied Mathematics (SIAM)},
}

@Article{Ioku2018,
  author    = {Ioku, Norisuke},
  journal   = {Mathematische Annalen},
  title     = {Attainability of the best Sobolev constant in a ball},
  year      = {2018},
  issn      = {1432-1807},
  month     = nov,
  number    = {1–2},
  pages     = {1--16},
  volume    = {375},
  doi       = {10.1007/s00208-018-1776-7},
  publisher = {Springer Science and Business Media LLC},
}

@Book{Hansen2004,
  author    = {Hansen, Carl J. and Kawaler, Steven D. and Trimble, Virginia},
  publisher = {Springer New York},
  title     = {Stellar Interiors: Physical Principles, Structure, and Evolution},
  year      = {2004},
  isbn      = {9781441991102},
  doi       = {10.1007/978-1-4419-9110-2},
  issn      = {0941-7834},
  journal   = {Astronomy and Astrophysics Library},
}

@Article{Amadori2017,
  author    = {Amadori, Anna Lisa and Gladiali, Francesca},
  journal   = {Nonlinear Analysis},
  title     = {Nonradial sign changing solutions to {L}ane–{E}mden problem in an annulus},
  year      = {2017},
  issn      = {0362-546X},
  month     = may,
  pages     = {294--305},
  volume    = {155},
  doi       = {10.1016/j.na.2017.02.027},
  publisher = {Elsevier BV},
}

@Article{Dancer2012,
  author    = {Dancer, Edward N. and Guo, Zongming and Wei, Juncheng},
  journal   = {Indiana University Mathematics Journal},
  title     = {Non-radial singular solutions of {L}ane-{E}mden equations in ${R}^{N}$},
  year      = {2012},
  issn      = {0022-2518},
  number    = {5},
  pages     = {1971--1996},
  volume    = {61},
  doi       = {10.1512/iumj.2012.61.4749},
  publisher = {Indiana University Mathematics Journal},
}

@Article{Talenti1976,
  author    = {Talenti, Giorgio},
  journal   = {Annali di Matematica Pura ed Applicata},
  title     = {Best constant in {S}obolev inequality},
  year      = {1976},
  issn      = {1618-1891},
  month     = dec,
  number    = {1},
  pages     = {353--372},
  volume    = {110},
  doi       = {10.1007/bf02418013},
  publisher = {Springer Science and Business Media LLC},
}

@Book{Folland1999,
  author    = {Folland, Gerald B.},
  publisher = {Wiley},
  title     = {Real analysis},
  year      = {1999},
  address   = {New York},
  edition   = {2. ed.},
  isbn      = {9780471317166},
  note      = {"A Wiley-Interscience publication." - Includes bibliographical references and index},
  series    = {Pure and applied mathematics},
  pagetotal = {386},
  ppn_gvk   = {24940365X},
  subtitle  = {Modern techniques and their applications},
}

@Article{Singh2014,
  author    = {Singh, Randhir and Nelakanti, Gnaneshwar and Kumar, Jitendra},
  journal   = {Proceedings of the National Academy of Sciences, India Section A: Physical Sciences},
  title     = {Approximate Solution of Two-Point Boundary Value Problems Using Adomian Decomposition Method with {G}reen’s Function},
  year      = {2014},
  issn      = {2250-1762},
  month     = dec,
  number    = {1},
  pages     = {51--61},
  volume    = {85},
  doi       = {10.1007/s40010-014-0170-4},
  publisher = {Springer Science and Business Media LLC},
}

@Article{Tsitouras2019,
  author    = {Tsitouras, Charalampos},
  journal   = {Applied Mathematics and Computation},
  title     = {Explicit {R}unge–{K}utta methods for starting integration of {L}ane–{E}mden problem},
  year      = {2019},
  issn      = {0096-3003},
  month     = aug,
  pages     = {353--364},
  volume    = {354},
  doi       = {10.1016/j.amc.2019.02.047},
  publisher = {Elsevier BV},
}

@article {Kajikiya2012,
    AUTHOR = {Kajikiya, Ryuji},
     TITLE = {Non-radial least energy solutions of the generalized {H}\'enon
              equation},
   JOURNAL = {J. Differential Equations},
  FJOURNAL = {Journal of Differential Equations},
    VOLUME = {252},
      YEAR = {2012},
    NUMBER = {2},
     PAGES = {1987--2003},
      ISSN = {0022-0396,1090-2732},
   MRCLASS = {35J91 (35A01 35B09 35J20 35J25)},
  MRNUMBER = {2853568},
MRREVIEWER = {Shengbing\ Deng},
       DOI = {10.1016/j.jde.2011.08.032},
       URL = {https://doi.org/10.1016/j.jde.2011.08.032},
}

\end{document}